\documentclass[10pt, a4paper]{amsart}
\usepackage[utf8]{inputenc}
\usepackage[top=1.2in, bottom=1.2in, left=1.2in, right=1.2in]{geometry}

\usepackage{eulervm}
\usepackage{amsmath, amsthm, amssymb, amsfonts}
\usepackage{arydshln}

\usepackage{mathtools}
\usepackage{mathrsfs}
\usepackage{graphicx}
\usepackage{enumitem}
\usepackage{xcolor}
\usepackage{hyperref}
\hypersetup{
	colorlinks=true,
	linkcolor=blue,
	filecolor=blue,
	citecolor = blue,      
	urlcolor=blue,
}
\usepackage{multirow}

\setlist[enumerate,1]{label={\upshape(\roman*)}}

\theoremstyle{plain}
\newtheorem{theorem}{Theorem}[section]
\newtheorem{prop}[theorem]{Proposition}
\newtheorem{corollary}[theorem]{Corollary}
\newtheorem{lemma}[theorem]{Lemma}

\theoremstyle{definition}
\newtheorem{definition}[theorem]{Definition}
\newtheorem{example}[theorem]{Example}
\newtheorem{remark}[theorem]{Remark}
\newtheorem{notation}[theorem]{Notation}

\numberwithin{equation}{section}

\renewcommand{\phi}{\varphi}
\renewcommand{\epsilon}{\varepsilon}
\newcommand\rom{\mathrm}
\newcommand{\kk}{\pmb{k}}

\title{Modules of logarithmic derivations in weighted projective spaces and applications to free divisors}
\thanks{J.M. and W.N. are partially supported by Grant PID2024-156181NB-C33 (funded by MICIU/AEI/10.13039/501100011033
	and by FEDER, UE) and Grant E22\_23R Álgebra y Geometría (funded by Diputación General de Aragón). J.M. is also supported by Grant CNS2024-154271 (funded by MICIU/AEI/10.13039/501100011033),
	Grant RYC2021-034300-I (funded by MICIU/AEI/10.13039/501100011033 and
	the European Union NextGenerationEU/PRTR), and
	Grant FQM-333 (funded by Junta de Andalucía).}

\author[J.~Martín-Morales]{Jorge Martín-Morales}
\address{
	Jorge Martín-Morales. Departamento de Matemáticas, IUMA \\
	Universidad de Zaragoza \\ 
	C.~Pedro Cerbuna 12 \\ 
	50009 Zaragoza, Spain} 
\urladdr{\url{http://riemann.unizar.es/~jorge}}
\email{\href{mailto:jorge@unizar.es}{jorge.martin@unizar.es}}

\author[W.~Ng Kwing King]{Wayne Ng Kwing King}
\address{
	Wayne Ng Kwing King. Universit\'e de Pau et des Pays de l'Adour,
	LMAP-UMR CNRS 5142, \\
	Avenue de l'Universit\'e -BP 1155-\\
	64013 Pau Cedex, France\\
	And\\
	\newline
	Departamento de Matemáticas, IUMA \\
	Universidad de Zaragoza \\ 
	C.~Pedro Cerbuna 12 \\ 
	50009 Zaragoza, Spain} 
\email{\href{mailto:wnkking@univ-pau.fr}{wnkking@univ-pau.fr}}

\begin{document}

\maketitle

\begin{abstract}
	We introduce a weighted version of the module of logarithmic derivations of a divisor in weighted projective space, and provide a 
	generalization of Saito's criterion for freeness in terms of weighted multiple eigenschemes (wME-schemes). Freeness of the non-standard $\mathbb{Z}$-graded module allows one to consider big families of free divisors in 
	affine and standard projective space, i.e. when the module of logarithmic derivations of the divisor is free over the respective coordinate rings. We present a method to identify and construct these new families of free 
	divisors in affine and projective space in any dimension, and give numerous explicit examples.
\end{abstract}

\section{Introduction}

The area of multigraded commutative algebra in connection with toric varieties has been very much active in the past decades: classical work in commutative algebra and geometry on projective spaces have been generalized to
 toric varieties \cite{CoxCoordringofToricVar95, Mustata_ToricVar02, HerMustataPayne10}, results such as weighted \cite{MR2048221} and multigraded \cite{MS04} Castelnuovo-Mumford regularity, weighted 
\cite{MR4945952} and multigraded \cite{MR4801838} generalization of Green’s linear syzygy theorem, and a wide range of works on multigraded syzygies 
\cite{MR2073194, MR2278757, MR2429452, MR4156411, MR4810076, MR4457395, MR3331930, MR4215648} have emerged and found many new applications.

The present work generalizes the study of freeness of divisors in projective spaces to a particular type of toric varieties: weighted projective spaces.
From standard $\mathbb{Z}$-graded polynomial ring of projective space to the nonstandard $\mathbb{Z}$-graded setting, we study the nonstandard $\mathbb{Z}$-graded module of logarithmic derivations tangent to a 
reduced divisor in weighted projective space. Geometrically, these are vector fields tangent to the divisor in a complex variety, a classical object of study in algebra and geometry; the seminal work of Saito \cite{Saito1980} 
introduced the sheaf of logarithmic vector fields and its dual, and gave a determinantal criterion that characterized free divisors, i.e. those such that their module of
logarithmic derivations admits a basis. Faenzi, Jardim and Montoya \cite{faenzi2024logarithmicvectorfieldsfoliations} developed a toric version of the sheaf of logarithmic vector fields for simplicial toric 
varieties and of Saito's criterion. Weighted projective space fits in this theory and we make ample use of these results applied to our setting.

On the other hand, a reformulation of Saito's criterion in terms of eigenschemes first in \cite{JV_Eigenscheme2024} over the projective plane turned out to be a useful tool for studying freeness. This was generalized over 
$n$-dimensional projective spaces in terms of Multiple eigenschemes (ME-schemes) in \cite{digennaro2025saitostheoremrevisitedapplication}. Extending this theory, we introduce weighted versions of eigenschemes and 
Multiple eigenschemes (wME-schemes). Our main result is Theorem \ref{propMEscheme}, a weighted version of \cite[Proposition $3.3$]{digennaro2025saitostheoremrevisitedapplication} building on 
\cite[Theorem $2.5$]{JV_Eigenscheme2024}, Saito's criterion reformulation in terms of eigenschemes.

From this, exploiting the nonstandard  $\mathbb{Z}$-graded setting of weighted projective spaces, we are able to build, as applications, new free divisors from old, not only 
in weighted projective spaces, but also giving the corresponding affine free divisors and build new free divisors in standard projective spaces of arbitrary dimensions via a construction called the \textit{cone construction}; 
we also give their free exponents. We would also like to mention work by others, especially in \cite{st2014}, on weighted homogeneous polynomials and their homogenization in regard to freeness.

A variety of new examples of free divisors are presented as applications:
\begin{enumerate}

	\item[(1)] Divisors in $\mathbb{A}^\ell$ and $\mathbb{P}^\ell$ coming from complete reflection arrangements:
	
	Complete reflection arrangements $\mathcal{A}^\ell_\ell(n) = \{ x_1 \ldots x_\ell \prod_{1 \le i < j \le \ell} (x_i^n - x_j^n) = 0 \}$ as studied in \cite[Proposition $6.77$]{orlik_terao_1992arrangements} are classical 
	examples of free divisors for all $n \ge 1$, with free exponents $(1, n+1, 2n+1, \ldots, (\ell-1)n+1)$. The non-complete arrangements $\mathcal{A}^k_\ell(n) = \{ x_1 \ldots x_k \prod_{1 \le i < j \le \ell} (x_i^n - x_j^n) = 0 \}$ 
	for $0 \le k \le \ell$ are also known to be free with exponents $(1, n+1, 2n+1, \ldots, (\ell-2)n+1, (\ell-1)n - \ell + k+1)$ in \cite[Proposition $6.85$]{orlik_terao_1992arrangements}.

	\begin{itemize}
		\item[\textbf{Aff}:] 
		Using similar notations, we show in Corollary \ref{corAffremovelines} that the divisor $\mathcal{D}_\ell^k(\{n_i\})$ for $0 \le k \le \ell$ in $\mathbb{A}^\ell$
		defined by 
		$$F_k = p \cdot \prod_{1 \le i < j \le \ell}(x_i^{n_i} - x_j^{n_j}),\ \ \text{where} \ p = 
		\begin{cases}
			1, & \text{if}\  k = 0\\
			x_1 \cdots x_k, & \text{if}\ 1 \le k \le \ell
		\end{cases}$$
		is free for any positive integers $n_1, \ldots, n_\ell.$

		\item[\textbf{Proj}:] 
		By the cone construction, denoting $\_^h$ homogenization in the variable $x_0$ in $\kk[x_0, \ldots, x_\ell]$, we show in Theorem \ref{corProjremovelines} that the divisor $\widetilde{\mathcal{D}}_\ell^k(\{n_i\})$ 
		for $0 \le k \le \ell$ in $\mathbb{P}^\ell$
		defined by
		$$x_0 \cdot F_k^h = x_0 p \cdot \prod_{1 \le i < j \le \ell}(x_i^{n_i}x_0^{n_j - n_i} - x_j^{n_j}),
		\ \ \text{where} \ p = 
		\begin{cases}
			1, & \text{if}\  k = 0\\
			x_1 \cdots x_k, & \text{if}\ 1 \le k \le \ell
		\end{cases}$$ 
		is free with exponents 
		$$(1, n_\ell + 1, n_\ell + n_{\ell-1} + 1, ... ,n_\ell + n_{\ell-1} + \cdots +n_3 + 1, n_\ell + n_{\ell-1} + \cdots +n_2 + 1- \ell + k)$$
		for positive integers $ n_1, \ldots, n_\ell$ such that $n_1 \le n_2 \le \ldots \le n_\ell.$
	\end{itemize}

\item[(2)] A variant of Brieskorn-Pham polynomials in \cite[Example $5.3$]{Buchweitz_Conca_2012}:

	In \cite[Example $5.3$]{Buchweitz_Conca_2012}, Buchweitz and Conca showed that for $G_j = x_1^{r_1} + \cdots + x_j^{r_j}$ with $j = 2, \ldots, i$ and any positive integers $r_1, \ldots, r_i$,
	the product $G_2 \cdots G_i$ of Brieskorn-Pham polynomials is a free divisor as an application of \cite[Proposition $5.1$]{Buchweitz_Conca_2012} on triangular free divisors.

	\begin{itemize}
		\item[\textbf{Aff}:] Using weighted Multiple eigenscheme techniques (Theorem \ref{propMEscheme}), we show in Corollary \ref{corBrieskorn-Pham} a variant of \cite[Example $5.3$]{Buchweitz_Conca_2012}.
		Let $\Lambda$ be any finite set of distinct non-zero elements over the base field $\kk$. The divisor in $\mathbb{A}^3$ defined by the polynomial
		$$F_\Lambda = (x^{r_0} + y^{r_1})\prod_{\alpha \in \Lambda} (x^{r_0} + y^{r_1} + \alpha z^{r_2})$$
		in $\kk[x,y,z]$ is free for all positive integers $r_0, r_1, r_2.$
		
		\item[\textbf{Proj}:] 
		\begin{enumerate}
			\item An immediate result is the homogeneous case of weighted projective free divisors: in $\mathbb{P}^2$, the divisor defined by
			$$F = (x^{r} + y^{r})\prod_{\alpha \in \Lambda} (x^{r} + y^{r} + \alpha z^{r})$$ is free for all $r \in \mathbb{N}^*$ with exponents $(r-1, r|\Lambda|)$ where $|\Lambda|$ is the cardinality of the set.

			\item By the cone construction, as earlier if we denote $\_^h$ homogenization in the variable $t$ in $\kk[x,y,z,t]$ and let $\Lambda$ and $F_\Lambda$ be as defined above, we show in 
			Theorem \ref{thmBrieskorn-PhamProj}
			that the divisor in $\mathbb{P}^3$ defined by the homogeneous polynomial
			$$t \cdot F_\Lambda^h$$
			is free with exponents $(1, \rom{max}\{r_0,r_1\} - 1, |\Lambda|\cdot \rom{max}\{r_0,r_1, r_2\} ).$
			
			In particular, if $r_0 \le r_1 \le r_2$, then the divisor defined in $\mathbb{P}^3$ by
			$$ t (t^{r_1-r_0}x^{r_0} + y^{r_1})\prod_{\alpha \in \Lambda} (t^{r_2-r_0}x^{r_0} + t^{r_2-r_1}y^{r_1} + \alpha z^{r_2})$$
			is free with exponents $(1,r_1 - 1, |\Lambda|r_2)$.
		\end{enumerate}
	\end{itemize}

\item[(3)] Generalized case of \cite[Theorem $4.2$]{digennaro2025saitostheoremrevisitedapplication} on pencils of hypersurfaces in $\mathbb{P}^n$ :

In \cite[Theorem $4.2$]{digennaro2025saitostheoremrevisitedapplication}, the following family of free divisors in $\mathbb{P}^n$ is given:
starting with the free hyperplane arrangement $\mathcal{A} : x_0 \cdots x_n = 0$, one considers its Jacobian ideal
$$J = (h_0, h_1, \ldots, h_n),$$
where $h_i := x_0 \cdots \hat{x_i}\cdots x_n$, for $0 \le i \le n$, with singular locus the ${{n+1}\choose{2}} $ codimension two
faces of the hypertetrahedron. Fix $m_0$ such that $n = 2m_0 + \epsilon$ where $\epsilon = 0$ or $1$. Then the defining polynomials of two hypersurfaces
$S_1$ and $S_2$ are given respectively as
$$f_1 = \sum_{i = m_0+1}^n h_i \text{ and } f_2 = \sum_{i = 0}^{m_0} h_i.$$
Furthermore, consider the pencil of hypersurfaces $C(f_1,f_2) = \{ S_{a,b} = a S_1 + b S_2\}_{[a;b] \in \mathbb{P}^1}$ of degree $n$.

In \cite[Theorem $4.2$]{digennaro2025saitostheoremrevisitedapplication}, Di Gennaro and Miró-Roig showed that the divisors 
\begin{enumerate}
	\item $S_1S_2$ defined by $f_1f_2$ is free with exponents $(1, 2,\cdots, 2)$, and
	\item $S_1S_2 \prod_{i = 3}^k a_i S_1 + b_i S_2$, defined by $f_1f_2\prod_{i = 3}^k (a_i f_1 + b_i f_2)$, where $[a_i;b_i] \in  \mathbb{P}^1$, $k \ge 3$
	and $a_i S_1 + b_i S_2$ are generic members of the pencil $C(f_1,f_2)$, is free with exponents $(2,\cdots, 2, n(k- 2) + 1)$.\\
\end{enumerate}

\begin{itemize}
	\item[\textbf{Aff}:] Now define
	$$\widetilde{h_i} := x_0^{r_0} \cdots \widehat{x_i^{r_i}}\cdots x_n^{r_n},$$
	where $r_i \in \mathbb{Z}_{>0}$ are any positive integers for $0 \le i \le n$, and for 
		\textit{any} choice of $m$
	 where $0 \le m \le n-1$, define
	$$\widetilde{f_1} = \sum_{i = m+1}^n \widetilde{h_i} \text{ and } \widetilde{f_2} = \sum_{i = 0}^{m} \widetilde{h_i}.$$
	For a polynomial $f$, denote $f^{\rom{red}}$ its reduced polynomial, i.e. product of square-free irreducible factors.
	Similarly, for a divisor $D$, denote $D^{\rom{red}}$ its corresponding reduced divisor.\\
	
	We show in Corollary \ref{corAff_f1f2} that
	\begin{enumerate}
		\item the divisor $(\widetilde{S_1}\widetilde{S_2})^{\rom{red}}$ in affine space $\mathbb{A}^{n+1}$ defined by the polynomial
		$$(\widetilde{f_1}\widetilde{f_2})^{\rom{red}} = x_0 \cdots x_n 
		\left( \sum_{i = 0}^m x_{0}^{r_{0}} \cdots \widehat{x_i^{r_i}}\cdots x_m^{r_m} \right)
		\left( \sum_{i = m+1}^n x_{m+1}^{r_{m+1}} \cdots \widehat{x_i^{r_i}}\cdots x_n^{r_n} \right)$$
		is free for any positive integers $r_0, \ldots, r_n$.
		\item The divisor $(\widetilde{S_1}\widetilde{S_2})^{\rom{red}} \prod_{i = 3}^k a_i \widetilde{S_1} + b_i \widetilde{S_2}$ in affine space $\mathbb{A}^{n+1}$ 
		defined by the polynomial
		$$\left(\widetilde{f_1}\widetilde{f_2}\right)^{\rom{red}}\prod_{i = 3}^k (a_i \widetilde{f_1} + b_i \widetilde{f_2}),$$
		for $k \ge 3$, where $a_i \widetilde{S_1} + b_i \widetilde{S_2} = V(a_i \widetilde{f_1} + b_i \widetilde{f_2}),$ with $[a_i;b_i] \in  \mathbb{P}^1$ for 
		$\ 3 \le i \le k$ are $k-2$ hypersurfaces in $\mathbb{A}^{n+1}$ corresponding to general distinct points of $L_{ \widetilde{f_1}, \widetilde{f_2}}$, is
		free for any positive integers $r_0, \ldots, r_n$.		
	\end{enumerate}

	\item[\textbf{Proj}:]
	
	\begin{enumerate}
		\item An immediate consequence of weighted projective free divisors is to consider the homogeneous case: let 
		$$h'_i := x_0^r \cdots \widehat{x_i^r}\cdots x_n^r,$$ 
		for $0 \le i \le n$ and $r \in \mathbb{Z}_{>0}$, and for \textit{any} choice of $m$ where $0 \le m \le n-1$, we define 
		$$f'_1 = \sum_{i = m+1}^n h'_i \text{ and } f'_2 = \sum_{i = 0}^{m} h'_i,$$
		the two homogeneous polynomials of degree $rn$ of the hypersurfaces $S'_1$ and $S'_2$ respectively.\\
		
		In $\mathbb{P}^n$, with the above notation, we show in Corollary \ref{corf1f2powerN} that the divisors
		\begin{enumerate}
			\item $(S'_1S'_2)^{\rom{red}}$ defined by 
			$$(f'_{1}f'_{2})^{\rom{red}} = x_0 \cdots x_n 
			\left( \sum_{i = 0}^m x_{0}^{r} \cdots \widehat{x_i^{r}}\cdots x_m^{r} \right)
			\left( \sum_{i = m+1}^n x_{m+1}^{r} \cdots \widehat{x_i^{r}}\cdots x_n^{r} \right)$$
			is free with exponents $(1, r+1,\cdots, r+1)$, and
			\item $(S'_1S'_2)^{\rom{red}} \prod_{i = 3}^k (a_i S'_1 + b_i S'_2)$, defined by $$(f'_1f'_2)^{\rom{red}}\prod_{i = 3}^k (a_i f'_1 + b_i f'_2),$$ for $k \ge 3$, $[a_i;b_i] \in  \mathbb{P}^1$,
			and $a_i S'_1 + b_i S'_2$ being generic members of the pencil $C(f'_1,f'_2)$, is free with exponents $(r+1,\cdots, r+1, rn(k - 2) + 1)$.
		\end{enumerate}

		\item In $\mathbb{P}^{n+1}$, by the cone construction, if we denote $\_^h$ homogenization in the variable $x_{n+1}$ in $\kk[x_0, \ldots, x_{n+1}]$ for a polynomial and its corresponding divisor, and if we let 
		the hyperplane $H = V(x_{n+1}) \subset \mathbb{P}^{n+1},$ we show in Theorem \ref{thmDiGProj} that
		for positive integers $r_0, \ldots, r_n$ such that $r_0 \le \ldots \le r_n$,
		\begin{enumerate}
			\item The divisor $H((\widetilde{S_1}\widetilde{S_2})^{\rom{red}})^h$ in projective space $\mathbb{P}^{n+1}$ defined by the
			homogeneous polynomial $x_{n+1} ((\widetilde{f_1}\widetilde{f_2})^{\rom{red}})^h =$
			$$x_0 \cdots x_{n+1}
			\left( \sum_{i = 0}^m x_{n+1}^{r_i - r_0} x_{0}^{r_{0}} \cdots \widehat{x_i^{r_i}}\cdots x_m^{r_m} \right)
			\left( \sum_{i = m+1}^n x_{n+1}^{r_i - r_{m+1}} x_{m+1}^{r_{m+1}} \cdots \widehat{x_i^{r_i}}\cdots x_n^{r_n} \right)$$
			is free with exponents $(1,1,r_1 +1, r_2+1, \ldots,  \widehat{r_{m+1} +1}, \ldots, r_n +1)$.
			
			\item The divisor $H((\widetilde{S_1}\widetilde{S_2})^{\rom{red}})^h \prod_{i = 3}^k \left( a_i \widetilde{S_1} + b_i \widetilde{S_2} \right)^h$ in projective space $\mathbb{P}^{n+1}$ defined by the 
			homogeneous polynomial
			$$x_{n+1} ((\widetilde{f_1}\widetilde{f_2})^{\rom{red}})^h\prod_{i = 3}^k \left( a_i \widetilde{f_1} + b_i \widetilde{f_2} \right)^h,$$
			for $k \ge 3$, $[a_i;b_i] \in  \mathbb{P}^1$ and 
			$\left( a_i \widetilde{S_1} + b_i \widetilde{S_2} \right)^h= V \left(\left( a_i \widetilde{f_1} + b_i \widetilde{f_2} \right)^h \right)$ are generic members
			of the pencil $C(x_{n+1}^{\rom{deg} \widetilde{f_2} - \rom{deg} \widetilde{f_1}} \cdot \widetilde{f_1}^h, \widetilde{f_2}^h )$, is free with
			exponents $$(1,r_1 +1, r_2+1, \ldots,  \widehat{r_{m+1} +1}, \ldots, r_n +1, (k-2)\sum_{i = 1}^n r_i + 1 ).$$
		\end{enumerate}
	\end{enumerate}
	
\end{itemize}

\end{enumerate}

\subsection*{Outline} In Section \ref{secPrelim}, we recall the background on the module of logarithmic derivations, eigenschemes and Multiple eigenschemes in projective space. 
In Section \ref{secwDerMod}, we introduce the 
module of logarithmic derivations in weighted projective space and how it fits into the theory of sheafs of logarithmic vector fields on simplicial toric varieties in \cite{faenzi2024logarithmicvectorfieldsfoliations}. 
In Section \ref{secwME}, we study weighted eigenschemes and Multiple eigenschemes (wME-schemes), proving Theorem \ref{propMEscheme}, Saito's criterion reformulation in terms of wME-schemes. 
Section \ref{secwFreeDiv} studies free divisors in weighted projective spaces and how they relate to the freeness of their corresponding divisors in affine and projective spaces, which for the latter via the cone 
construction. In Section \ref{secApplis}, from finding new free divisors from old to proving freeness with the new techniques, new examples of free divisors are provided as applications, in both affine and 
projective spaces.

\subsection*{Acknowledgments} The authors gratefully acknowledge the laboratory of Algebraic Geometry and Applications to Information Theory (French acronym: GAATI), where an important part of this 
research was carried out, and would like to thank the second author's advisor Jean Vallès for useful discussions on the subject.
The computer algebra systems \texttt{Singular} \cite{DGPS} and \texttt{Macaulay2} \cite{M2} were indispensable in computing examples.

\section{Preliminaries} \label{secPrelim}

\textbf{Notations.} 
In this section, let $R = \kk[x_0, \ldots, x_n]$ be the ring of polynomials in $n+1$ variables with the standard grading, where $\kk$ will be a field of characteristic $0$ although many statements can be 
generalized to fields of any characteristic. When the ring of polynomials will have a non-standard $\mathbb{Z}$-grading given by some 
weight vector $\omega = (\omega_0, \ldots, \omega_n)$, then we will denote it by $S= \kk[x_0, \ldots, x_n]$. 
The partial derivatives with respect to the variables $x_0, \ldots, x_n$ will be written as ${\partial_{x_i}} := \frac{\partial}{\partial x_i}$.

\subsection{Module of logarithmic derivations}
The polynomial ring $R = \bigoplus_{i} R_i$ is a 
standard graded ring over the field $R_0 = \kk$, hence $\rom{Proj}(R) = \mathbb{P}^n$. The module of $\kk$-derivations over $R$, $\rom{Der}_{\kk}(R) \cong \bigoplus_{i=0}^n R \,\partial_{x_i}$ is a free 
graded $R$-module of rank $n+1$.

For a reduced divisor $D = V(f)$ where $f \in R_d$ is a reduced homogeneous polynomial of degree $d \ge 1$, the module of derivations $\rom{Der}_{\kk}(-\rom{log}\, D)$ tangent to $D$ is defined as
\begin{equation} \label{eqStandardDerLog}
	\rom{Der}_{\kk}(-\rom{log}\, D) := \{ \delta \in \rom{Der}_{\kk}(R) \, | \, \delta(f) \in Rf \}.
\end{equation}
It is a reflexive $R$-module (\cite[Corollary 1.7]{Saito1980}) and is a graded submodule of $\rom{Der}_{\kk}(R)$ in the natural way:
$\rom{Der}_{\kk}(-\rom{log}\, D)_m = \rom{Der}_{\kk}(-\rom{log}\, D) \cap \rom{Der}_{\kk}(R)_m \text{ for }m \ge 0.$

The Euler derivation
$\delta_\mathcal{E} = \sum^n_{i = 0} x_i \,\partial_{x_i}$
lies in $\rom{Der}_{\kk}(-\rom{log}\, D)_1$ and satisfies $\delta_\mathcal{E}(f) = df$. When char($\kk$) does not divide deg($f$), we obtain the following decomposition for any 
$\delta \in \rom{Der}_{\kk}(-\rom{log}\, D)$:
$$\delta = \widetilde{\delta} + (\frac{\delta(f)}{df})\delta_\mathcal{E},$$ 
where $\widetilde{\delta} = \delta - (\frac{\delta(f)}{df} )\delta_\mathcal{E}$ such that $\widetilde{\delta}(f) = 0$. This yields the decomposition
\begin{equation} \label{eqDer0}
	\rom{Der}_{\kk}(-\rom{log}\, D) = R\delta_\mathcal{E} \oplus \rom{Der}_0(-\rom{log}\, D).
\end{equation}

$\rom{Der}_0(-\rom{log}\, D) = \{\delta \in \rom{Der}_{\kk}(R) \, | \, \delta(f) = 0 \}$ is called the \textit{module of logarithmic derivations} and is also the kernel of the \textit{gradient} map 
\begin{equation}
	\nabla f = (\frac{\partial f}{\partial x_0}, \ldots, \frac{\partial f}{\partial x_n}) : R^{n+1} \to R(d-1),
\end{equation}
which surjects onto $J_D(d-1)$. Here $J_D = ( \partial_{x_0} f, \ldots, \partial_{x_n} f)$ is the Jacobian ideal of $f$, which defines the singular locus of the divisor $D$.
This fits in a short exact sequence
$$ 0 \to \rom{Der}_0(-\rom{log}\, D) \to R^{n+1} \xrightarrow{\nabla f} J_D(d-1) \to 0,$$
where $\rom{Der}_0(-\rom{log}\, D)$ is the first syzygy module of $J_D(d-1)$ and with a shift in degree, this is the start of a resolution of the Jacobian ideal $J_D$, 
making $\rom{Der}_0(-\rom{log}\, D) \simeq Syz(J_D)$, 
the first syzygy module of $J_D$. By decomposition (\ref{eqDer0}), $\rom{Der}_0(-\rom{log}\, D)$ is also a reflexive module.

$\rom{Der}_{\kk}(-\rom{log}\, D)$ defined in (\ref{eqStandardDerLog}) is also called the \textit{extended (or reduced) module of logarithmic derivations}
(\cite[Definition $3.1$]{faenzi2024logarithmicvectorfieldsfoliations} and \cite[Chapter $1$]{faenzi_research}) and can be seen as the kernel of the map 
\begin{equation}\label{mapJacobianMod}
	\overline{\nabla f} = (\frac{\partial f}{\partial x_0}, \ldots, \frac{\partial f}{\partial x_n}) : R^{n+1} \to \frac{R}{(f)}(d-1).
\end{equation}
This fits in a short exact sequence
$$ 0 \to \rom{Der}_{\kk}(-\rom{log}\, D) \to R^{n+1} \xrightarrow{\overline{\nabla f}} \bar{J_D}(d-1) \to 0,$$
where $\bar{J_D}$ is the image of the Jacobian ideal in $R/(f)$. One can see the map (\ref{mapJacobianMod}) as 
\begin{align*}
	\rom{Der}_{\kk}(R)& \simeq R^{n+1} \to \frac{R}{(f)}(d-1),\\
	\delta& \mapsto \delta(f)\, \rom{mod}\, f,
\end{align*}
with the kernel,
a submodule of $\rom{Der}_{\kk}(R)$ (see \cite{saito_avatars_jv_df_24}).

\begin{definition}
	A reduced divisor $D$ in projective space $\mathbb{P}^n$ is \textit{free with exponents $(a_1, \ldots, a_n)$} if  
	$$\rom{Der}_0(-\rom{log}\, D) = R \delta_{1} \oplus \ldots \oplus R \delta_{n}  \cong \bigoplus^n_{i=1} R(-a_i)$$ is a free $R$-module, generated by a basis of $n$ of derivations $\delta_{1}, \ldots, \delta_{n}$, 
	with $ a_i = \rom{deg}(\delta_{i})$ for $1 \le i \le n$. By decomposition \eqref{eqDer0}, this is equivalent to $\rom{Der}_{\kk}(-\rom{log}\, D) \cong \bigoplus^n_{i=0} R(-a_i)$ being a free $R$-module with basis 
	consisting of the Euler derivation $\delta_{\mathcal{E}_\omega}$ and $\delta_{1}, \ldots, \delta_{n}$.
\end{definition}

\begin{remark}
	An equivalent definition of free divisor is that $J_D$, the Jacobian ideal of $f \in \kk[x_1, \ldots, x_n]$ is \textit{perfect} (or equivalently in this case \textit{Cohen-Macaulay} by 
	\cite[Theorem $2.1.5$]{Bruns_Herzog_1998}) of codimension $2$; perfect ideals of codim $2$ in this setting are completely described by the Hilbert-Burch theorem (see \cite[Theorem $20.15$]{Eisenbud95}).
\end{remark}

\subsection{Eigenscheme and Multiple eigenscheme}

The language of eigenscheme and Multiple eigenscheme (ME-scheme) in $\mathbb{P}^n$ has been expressed in \cite{digennaro2025saitostheoremrevisitedapplication} and 
\cite{Beorchia_Ternary_tensors_eigenscheme_2021} 
in terms of tensors, in particular \textit{partially symmetric tensors} which will be elements of $\rom{Sym}^d \kk^{n+1} \otimes \kk^{n+1}$, for $\kk$ a field of characteristic $0$. Once a basis has been chosen for 
$\kk^{n+1}$, then the
space $\rom{Sym}^d\kk^{n+1}$ is identified with $\kk[x_0, \ldots, x_n]_d$, and hence $\rom{Sym}^d \kk^{n+1} \otimes \kk^{n+1}$ will be the space of $(n+1)$-tuples of homogeneous polynomials of degree $d$, 
namely $(\rom{Sym}^d\kk^{n+1})^{\oplus n+1}$. We will also use the language of tensors to define weighted versions of the eigenscheme.

In the study of free divisors, a useful condition \cite[Theorem $2.5$]{JV_Eigenscheme2024} for testing freeness of curves in $\mathbb{P}^2$ was introduced; it reformulated Saito's criterion using eigenscheme as 
main tool. The idea was naturally extended to $\mathbb{P}^n$ in 
\cite{digennaro2025saitostheoremrevisitedapplication} where the Multiple eigenscheme was henceforth introduced. We recall the definitions:

Let $T = (g_0, g_1, \ldots, g_n) \in (\rom{Sym}^d\kk^{n+1})^{\oplus n+1}$ be a partially symmetric tensor. The \textit{eigenscheme} of $T$ is the closed subscheme $E(T) \subset \mathbb{P}^n$
defined by the $2 \times 2$ minors of the matrix
$$M_{T} = \begin{pmatrix}
	 x_0 & x_1 & \ldots& x_n\\
	g_0 & g_1 & \ldots &g_n
\end{pmatrix}.$$

For $1 \le i \le r \le n-1,$ let $T_i = (g_0^i, g_1^i, \ldots, g_n^i) \in (\rom{Sym}^d\kk^{n+1})^{\oplus n+1}$ be $r$ partially symmetric tensors. The \textit{Multiple eigenscheme} (ME-scheme)
of $T_1, \ldots, T_r$ is the closed subscheme $E(T_1,\ldots, T_r) \subset \mathbb{P}^n$ defined by the $(r+1) \times (r+1)$ minors of the matrix

$$M_{T_1, \ldots, T_{r}} = \begin{pmatrix}
	x_0 & x_1 & \ldots& x_n\\
	g_0^1 & g_1^1 & \ldots &g_n^1\\
	\vdots &&&\vdots\\
	g_0^r & g_1^r & \ldots &g_n^r
\end{pmatrix}.$$

By the Giambelli–Thom–Porteous formula (see \cite[Ch $12.1$]{eisenbud20163264}), the expected codimension of $E(T_1, \ldots, T_{r})$, $$\rom{codim}(I_{E(T_1, \ldots, T_{r})}) = n -r +1$$ 
when $T_1, \ldots, T_r$ are general, where
$I_{E(T_1, \ldots, T_{r})}$ denotes the ideal of maximal minors of $M_{T_1, \ldots, T_{r}}$. Hence the eigenscheme $E(T)$ is in general a finite scheme, and in particular the ME-scheme $E(T_1, \ldots, T_{n-1})$ 
is generally in codimesion $2$; these assumptions are both present in the main theorems  \cite[Theorem $2.5$]{JV_Eigenscheme2024} and \cite[Proposition $3.3$]{digennaro2025saitostheoremrevisitedapplication}.
 
 The indeterminancy locus and fixed points of the rational map 
 $$\mathbb{P}^n \dasharrow \mathbb{P}^n, p \mapsto  [g_0(p): \ldots: g_n(p)]$$
  define set theoretically the eigenscheme $E(T)$,
 while the Multiple eigenscheme $E(T_1,\ldots, T_r)$ is the closure of the union of points which are in the linear space spanned by  $[g^i_0(p): \ldots: g^i_n(p)]$ for $1 \le i \le r$. 
The coordinate rings of eigenschemes and Multiple eigenschemes that have the expected codimensions are Cohen-Macaulay rings, their defining ideals saturated and are standard determinantal schemes 
(see \cite{digennaro2025saitostheoremrevisitedapplication} and \cite{BEORCHIA2023107269}). We refer to \cite{abo2017a}, \cite{Beorchia_Ternary_tensors_eigenscheme_2021} for further details on 
eigenscheme of tensors.

\section{Module of logarithmic derivations for weighted projective spaces} \label{secwDerMod}

\subsection{Weighted projective space}

Let $S = \kk[x_0, \ldots, x_n]$ be a polynomial ring over a field $\kk$ and $\omega$ be the weight vector $(\omega_0, \ldots, \omega_n)$ where 
$\omega_i = \text{deg } x_i > 0$, hence $(S, \omega)$ or $S(\omega_0, \ldots, \omega_n)$ stands for a polynomial ring with the graduation 
given by $\omega$, and has a \textit{non}standard $\mathbb{Z}$-grading if not all $\omega_i = 1.$ The \textit{weighted projective space} 
denoted $\mathbb{P}^n_\omega$ or $\mathbb{P}(\omega_0, \ldots, \omega_n)$ is $\mathbb{P}^n_\omega := \rom{Proj }\, S$.

Every polynomial in $S$ is a sum of monomials $\textbf{x}^\textbf{r} = \prod_i x_i^{r_i}$, having weighted degree $\sum_i r_i \omega_i$. 

\begin{definition}
	The \textit{weighted degree of a polynomial} $g$, denoted $\rom{deg}_\omega(g)$, is the maximal weighted degree of its monomials.
	A polynomial $f$ is called \textit{weighted homogeneous} (or \textit{quasihomogeneous}) of degree $d$ or \textit{homogeneous of weight $d$} if every monomial of $f$ has weighted degree $d$. 
\end{definition}

\begin{lemma}(See e.g. \cite[Lemma $5.5, 5.7$]{Iano-Fletcher_2000}) \label{lemWProjSpaceProperties}
	\begin{enumerate}
		\item Let $a$ be a positive integer. Then  $\mathbb{P}(\omega_0, \ldots, \omega_n) \simeq \mathbb{P}(a\omega_0, \ldots, a\omega_n)$.
		\item Suppose $\omega_0, \ldots, \omega_n$ have no common factor. Let $q = gcd(\omega_0, \ldots, \widehat{\omega_i}, \ldots, \omega_n)$, 
		the greatest common factor of $\omega_j$ with $j \ne i$, then 
		
		$\mathbb{P}(\omega_0, \ldots, \omega_n) \simeq \mathbb{P}(\omega_0/q,\ldots, \omega_i, \ldots,  \omega_n/q)$.
	\end{enumerate}	
\end{lemma}

\begin{proof}
	These follow from the Proj construction, the associated graded rings being isomorphic, and uses the \textit{q}th 
	\textit{Veronese embedding}: for $S$ graded ring, define $S^{(q)} = \bigoplus_{j \ge 0} S_{qj} $ be the (truncation) subring having as 
	$j$th graded part $S_{qj}$. Then $\rom{Proj }\, S^{(q)} \simeq \rom{Proj }\, S$ are canonically isomorphic.
\end{proof}

$\mathbb{P}(\omega_0, \ldots, \omega_n)$ is then said to be \textit{well-formed} if 
$gcd(\omega_0, \ldots, \widehat{\omega_i}, \ldots, \omega_n) = 1$ for all $i \in \{ 0, \ldots, n\}$.

\begin{prop}(See e.g. \cite[Ch $3,4,5$]{cox2011toric}) \label{propWeightProjDivClassGrp}
	The weighted projective space $\mathbb{P}(\omega_0, \ldots, \omega_n)$ with $\rom{gcd}(\omega_0, \ldots, \omega_n) = 1$ is a simplicial normal toric variety and its divisor class group 
	$\rom{Cl}(\mathbb{P}_\omega) \simeq \mathbb{Z}$. Furthermore $\rom{Pic}(\mathbb{P}_\omega) \subseteq \rom{Cl}(\mathbb{P}_\omega)$ maps to the subgroup $m\mathbb{Z} \subseteq \mathbb{Z}$ 
	where $m = \rom{lcm}(\omega_0, \ldots, \omega_n).$
\end{prop}

\subsection{Module of logarithmic derivations for $\mathbb{P}^n_\omega$}

A reduced divisor $D = V(f) \subseteq \mathbb{P}^n_\omega$ in weighted projective space, where $f \in S_d$ a weighted homogeneous polynomial of degree $d$, has an associated reduced module of logarithmic derivations 
$$\rom{Der}_{\kk}(-\rom{log}\, D) = \{ \delta \in \rom{Der}_{\kk}(S) \, | \, \delta(f) \in Sf \},$$ 
which is similarly defined in (\ref{eqStandardDerLog}) over the standard graded ring $R$.
If the ring $S = \bigoplus_{i \ge 0} S_i$ has a nonstandard $\mathbb{Z}$-grading given by the weight vector $\omega$ of $\mathbb{P}^n_\omega$ (here we denote by $S = (S, \omega)$), then the graded free $S$-module 
$\rom{Der}_{\kk}(S)$ of derivations will have a nonstandard $\mathbb{Z}$-grading once we set the weighted degrees of the partial derivatives $\partial_{x_i}$ for $i = 0, 1, \ldots, n$. Once defined, this will naturally 
induce a nonstandard $\mathbb{Z}$-grading on the submodule
$\rom{Der}_{\kk}(-\rom{log}\, D)$. To achieve this, one needs to first 
consider the following canonical derivation:

\begin{definition} \label{defwEuler}
	The \textit{weighted Euler derivation} $\delta_{\mathcal{E}_\omega} = \sum^n_{i = 0} \omega_i x_i \partial_{x_i}$, associated to a weighted vector $\omega = (\omega_0, \ldots, \omega_n) \in \mathbb{Z}^{n+1}_{> 0}$, 
	satisfies that if
	$f \in S_d$ is any weighted homogeneous polynomial of degree $d$, then $\delta_{\mathcal{E}_\omega}(f) = df.$
\end{definition}

As an $S$-module, the module of derivations decomposes as $\rom{Der}_{\kk}(S) \cong \bigoplus^n_{i=0} S \, \partial_{x_i}$. 
One now makes a choice on the grading of the $S$-module $\rom{Der}_{\kk}(S)$:

\begin{definition} \label{defGradingDerS}
	\textit{The $\mathbb{Z}$-grading on the $S$-module $\rom{Der}_{\kk}(S)$} is given by the $\mathbb{Z}$-grading on the ring $S$ and 
	by the weighted degree of its generators, one defines as
	$$\rom{deg}_\omega(\partial_{x_i}) = 1 - \omega_i, \text{ where } 0 \le i \le n.$$
\end{definition}

Hence $\rom{Der}_{\kk}(S)$ decomposes as $\rom{Der}_{\kk}(S) \cong \bigoplus^n_{i=0} S(-v_i)$, where 
$v_i = \rom{deg}_\omega(\partial_{x_i})$. 

\begin{definition}
	Let $\rom{deg}_\omega(\delta) := \rom{max}\{\rom{deg}_\omega(g_i \partial_{x_i})\}_{0 \le i \le n},$ denote the \textit{weighted degree of a derivation} $\delta = \sum^n_{i=0} g_i \partial_{x_i} \in \rom{Der}_{\kk}(S)$  
	where 
	$\rom{deg}_\omega(g_i \partial_{x_i}) := \rom{deg}_\omega(g_i) + \rom{deg}_\omega(\partial_{x_i})$. 
	A derivation $\delta$ is \textit{weighted homogeneous of degree} $d$ if $\forall \, i \in \{0, \ldots, n\}$, $g_i$ is weighted homogeneous
	and $\rom{deg}_\omega(g_i \partial_{x_i}) = d$.
\end{definition}

\begin{remark}
	Throughout this paper and for developing the theory, we will assume the $\mathbb{Z}$-graded structure on the $S$-module $\rom{Der}_{\kk}(S)$ in Definition \ref{defGradingDerS}. It is a natural choice because with the 
	standard weights 
	$\omega_i = 1$ for all $i$, one recovers the pdegree of a derivation (see \cite[Definition $4.2$]{orlik_terao_1992arrangements}). 
	Furthermore, note that the weighted Euler derivation is weighted
	homogeneous of degree $1$ since for $i = 0, \ldots, n$, 
	$\rom{deg}_\omega(\omega_i x_i \partial_{x_i}) = \rom{deg}_\omega(\omega_i x_i) + \rom{deg}_\omega(\partial_{x_i}) = \omega_i + (1 - \omega_i) = 1$.
\end{remark}

With this $\mathbb{Z}$-graded structure on $\rom{Der}_{\kk}(S)$, the module of logarithmic derivations 
$\rom{Der}_{\kk}(-\rom{log}\, D)$ becomes a graded submodule as follows
$$\rom{Der}_{\kk}(-\rom{log}\, D)_m = \rom{Der}_{\kk}(-\rom{log}\, D) \cap \rom{Der}_{\kk}(S)_m \text{ for }m \in \mathbb{Z}.$$

\begin{lemma}
	Let $D$ be a reduced divisor in weighted projective space. Then $\rom{Der}_{\kk}(-\rom{log}\, D)$ is a graded module:
	$$\rom{Der}_{\kk}(-\rom{log}\, D) = \bigoplus_{m \in \mathbb{Z}} \rom{Der}_{\kk}(-\rom{log}\, D)_m.$$
\end{lemma}

\begin{proof}
	($\supseteq$) is clear by definition.\\
	($\subseteq$) Conversely let us denote $\rom{Der}(f) := \rom{Der}_{\kk}(-\rom{log}\, D)$ where 
	the divisor $D = V(f)$ is defined by the weighted homogeneous polynomial $f$ of degree $d$.
	Let $\delta \in \rom{Der}(f)$, then $\delta$ can be written $\delta = \sum_j \delta_j$ where 
	$\delta_j = \sum_{i = 0}^n g_{j,i} \partial_{x_i}$ are weighted homogeneous of degree $j$. Since  
	$\delta \in \rom{Der}(f)$, we have $\delta(f) = hf$ for some polynomial $h \in S.$ We can further write 
	$hf = \sum_l h_l f$ where $h_l$ are weighted homogeneous of degree $l$.
	
	On the other hand, we have $\delta(f) = \sum_j \delta_j(f)$ where for each $j$, $\delta_j(f)$ is weighted homogeneous of
	degree $j + d - 1$ or $0$: indeed
	$\delta_j(f) = \sum_{i = 0}^n g_{j,i} \frac{\partial f}{\partial x_i}$ and since $\rom{deg}_\omega(g_{j,i} \partial_{x_i}) = j$, then 
	$\rom{deg}_\omega(g_{j,i}) = j - (1 - \omega_i)$, hence 
	$\rom{deg}_\omega(g_{j,i} \frac{\partial f}{\partial x_i}) = j - (1 - \omega_i) + (d-\omega_i) = j + d - 1$ or $0$.
	
	Hence 
	$$\delta(f) = \sum_l h_l f = \sum_j \delta_j(f)$$
	and comparing terms of same weighted degree, we have two cases:
	\begin{equation*}
		\begin{cases}
			\delta_j(f) = h_l f, &\text{if } \rom{deg}_\omega(h_l) = j-1 \\
			\delta_j(f) = 0, &\text{else } \rom{deg}_\omega(h_l) \ne j-1
		\end{cases}
	\end{equation*}
	In both cases, $\delta_j \in \rom{Der}(f)_j$ and hence $\delta = \sum_j \delta_j \in \bigoplus_{m \in \mathbb{Z}} \rom{Der}_{\kk}(-\rom{log}\, D)_m$ as desired.
\end{proof}

$\rom{Der}_{\kk}(-\rom{log}\, D)$ decomposes as
\begin{equation} \label{eqWDer0}
	\rom{Der}_{\kk}(-\rom{log}\, D) = S\delta_{\mathcal{E}_\omega} \oplus \rom{Der}_0(-\rom{log}\, D),
\end{equation}
as in the case of standard projective space, when char($\kk$) does not divide $\rom{deg}_\omega(f)$, because for any $\delta \in \rom{Der}_{\kk}(-\rom{log}\, D)$, one can always write
$\delta = \widetilde{\delta} + (\frac{\delta(f)}{df})\delta_{\mathcal{E}_\omega}$
with $\delta_{\mathcal{E}_\omega}$ weighted Euler derivation and $\widetilde{\delta} \in \rom{Der}_0(-\rom{log}\, D).$

For $f \in S_d$ a weighted homogeneous polynomial of degree $d$, its associated module of logarithmic 
derivations $\rom{Der}_0(-\rom{log}\, D) =  \{\delta \in \rom{Der}_{\kk}(S) \, | \, \delta(f) = 0 \}$ can be defined as the kernel of the
map 
\begin{equation} \label{defKerWDer0}
	\nabla f = (\frac{\partial f}{\partial x_0}, \ldots, \frac{\partial f}{\partial x_n}) : \bigoplus_{i=0}^n S(\omega_i -1) \to S(d-1),
\end{equation}
and $\rom{Der}_{\kk}(-\rom{log}\, D)$ as the kernel of 
\begin{equation} \label{defKerWDer}
	\overline{\nabla f} = (\frac{\partial f}{\partial x_0}, \ldots, \frac{\partial f}{\partial x_n}) : \bigoplus_{i=0}^n S(\omega_i -1) \to \frac{S}{(f)}(d-1).
\end{equation}

By a shift of one in degree, the sheafifications of the graded modules 
$$\rom{Der}_{\kk}(-\rom{log}\, D)(1)^{\sim} = \mathcal{T}_\Sigma \langle D \rangle \text{ and } \rom{Der}_{0}(-\rom{log}\, D)(1)^{\sim} = \mathcal{T}_\Sigma \langle D \rangle_0 $$ are precisely the 
\textit{extended toric logarithmic sheaf} and \textit{toric logarithmic sheaf} respectively defined in 
\cite{faenzi2024logarithmicvectorfieldsfoliations} in more generality for a divisor on a simplicial toric variety.

\begin{remark}
	The shift in degree of $1$ happens because in \cite[Section $3$]{faenzi2024logarithmicvectorfieldsfoliations}, 
	$\nabla f : \bigoplus_{1 \le i \le r} \mathcal{O}_X(D_i) \to \mathcal{O}_X(\beta)$ where $\beta = \rom{deg}(f) \in \rom{Cl}(X)$ 
	while in our definition of $\nabla f$ (\ref{defKerWDer0}), 
	we would have $\nabla f : \bigoplus_{0 \le i \le n} \mathcal{O}_{\mathbb{P}^n_\omega}(-\rom{deg}_\omega(\partial_{x_i})) \to \mathcal{O}_{\mathbb{P}^n_\omega}(\rom{deg}_\omega(f)-1)$ 
	where $\rom{deg}_\omega(f)-1 \in \rom{Cl}(\mathbb{P}^n_\omega) \simeq \mathbb{Z}$, and similarly for $\overline{\nabla f}$.
\end{remark}

Hence by Proposition \ref{propWeightProjDivClassGrp} and Lemma \ref{lemWProjSpaceProperties}, our theory developed in the particular case of weighted projective spaces fits into the theory 
\cite{faenzi2024logarithmicvectorfieldsfoliations} on simplicial toric varieties. For instance, our decomposition of $\rom{Der}_{\kk}(-\rom{log}\, D)$ (\ref{eqWDer0}) up to a shift of $1$ in degree is 
precisely \cite[Proposition $3.3$]{faenzi2024logarithmicvectorfieldsfoliations} applied to weighted projective space. We have that
$\rom{Der}_{\kk}(-\rom{log}\, D)$, hence also $\rom{Der}_0(-\rom{log}\, D)$ are reflexive modules (see e.g. \cite[Proposition $3.2$]{faenzi2024logarithmicvectorfieldsfoliations}).

\begin{definition} \label{defwFree}
	A reduced divisor $D$ in weighted projective space $\mathbb{P}^n_\omega$ is \textit{free with weighted exponents $(a_1, \ldots, a_n)$} if  
	$$\rom{Der}_0(-\rom{log}\, D) = S \delta_{1} \oplus \ldots \oplus S \delta_{n} \cong \bigoplus^n_{i=1} S(-a_i)$$ is a free $S$-module, generated by a basis of $n$ of derivations $\delta_{1}, \ldots, \delta_{n}$, 
	with $ a_i = \rom{deg}_\omega(\delta_{i})$ for $1 \le i \le n$. By decomposition \eqref{eqWDer0}, this is equivalent to $\rom{Der}_{\kk}(-\rom{log}\, D) \cong \bigoplus^n_{i=0} S(-a_i)$ being a 
	free $S$-module with basis 
	consisting of the weighted Euler derivation $\delta_{\mathcal{E}_\omega}$ and $\delta_{1}, \ldots, \delta_{n}$.
\end{definition}

\begin{notation}
	From now on, the modules of logarithmic derivations will also be denoted by $\rom{Der}(f) := \rom{Der}_{\kk}(-\rom{log}\, D)$ and $ \rom{Der}_0(f) := \rom{Der}_0(-\rom{log}\, D)$ of a divisor $D$ 
	defined by a polynomial $f$, which will be made clear from the context.
\end{notation}

\begin{remark} \label{rmkDerIntersection}
	For $f_1, f_2 \in S$ any two polynomials, we have the following useful property:
	\begin{equation} \label{eqDerIntersection}
		\rom{Der}(f_1f_2) = \rom{Der}(f_1) \cap \rom{Der}(f_2).
	\end{equation}
	In particular, note that $\rom{Der}(f^m) = \rom{Der}(f)$ for any polynomial $f$ of degree $\ge 1$ and $m \ge 1$.
	Because $\rom{Der}(f) = \rom{Der}(f^\rom{red})$, where $f^\rom{red}$ denotes the reduced polynomial of $f$, 
	one is essentially interested in \textit{reduced divisors} when studying modules of logarithmic derivations. 
\end{remark}

Let $\nabla (f) $ be the (column) vector of partial derivatives of $f$.
We now state a determinantal characterization of free divisors, \textit{Saito's criterion} \cite{Saito1980}, in some general form 
 for a reduced polynomial $f$ in a polynomial ring in $n$ variables: 
$\rom{Der}(f)$ is free if and only if there exists an $n \times n$ matrix 
$$M = \begin{pmatrix}
	g_1^1 & g_2^1 & \ldots &g_n^1\\
	\vdots &&&\vdots\\
	g_1^n & g_2^n & \ldots &g_n^n
\end{pmatrix}$$
 of tensors associated to $n$ derivations in $\rom{Der}(f)$, i.e. $M\nabla (f) = 0$ modulo $f$, such that
$\rom{det}(M) = f.$ Such a matrix (or its transpose) will be called a \textit{Saito matrix} or sometimes called a discriminant or coefficient matrix.
 
\begin{remark} \label{remExistsBasisHomog}
	In (weighted) projective space, if the divisor is free, then there always exists a basis of (weighted) homogeneous derivations for the graded free module $\rom{Der}(f)$ of rank $n$ 
	(see e.g. \cite[Theorem A.$20$]{orlik_terao_1992arrangements}),
	from which the free exponents can be deduced.
\end{remark}

\begin{remark}
	It can be seen directly but also as an immediate consequence of Saito's criterion that $f \in S$ remains a free divisor in any polynomial ring $S[X]$ over $S$, that is a cone on $D = V(f)$ is also free, as pointed 
	out in \cite[Remark $2.5$]{Buchweitz_Conca_2012} and \cite[Section $4.4$]{faenzi2024logarithmicvectorfieldsfoliations}.
\end{remark}

\section{Weighted Eigenschemes} \label{secwME}

In this section, we introduce the weighted versions of the eigenscheme and Multiple eigenscheme, which will be used to prove our version of \cite[Proposition $3.3$]{digennaro2025saitostheoremrevisitedapplication} 
in the weighted projective space $\mathbb{P}^n_\omega$ of weight vector $\omega = (\omega_0, \omega_1, \ldots, \omega_n)$.

We will consider tensors associated to weighted homogeneous derivations. A tensor $T$ associated to a weighted homogeneous derivation $\delta = \sum_{j=0}^n g_j \partial_{x_j}$ of degree $d$
will be identified as a tuple $T = (g_0, g_1, \ldots, g_n) \in (Sym \, k^{n+1})^{\oplus {n+1}}$ where $\rom{deg}_\omega(g_j) = d - \rom{deg}_\omega(\partial_{x_j}) = d - (1 - \omega_j)$ by Definition \ref{defGradingDerS}.

\begin{definition} \label{defwME-scheme}
	Let $T$ be a tensor associated to a weighted homogeneous derivation of degree $d$. Then $T = (g_0, g_1, \ldots, g_n) \in \bigoplus_{j = 0}^{n}(Sym^{d_{j}}k^{n+1}) \subset (Sym \, k^{n+1})^{\oplus {n+1}},$ 
	where $d_{j} = d - (1 - \omega_j)$. The \textit{weighted eigenscheme} of $T$ is the closed subscheme $E(T) \subset \mathbb{P}^n_\omega$
	defined by the $2 \times 2$ minors of the homogeneous polynomial matrix
	$$M_{T} = \begin{pmatrix}
		\omega_0 x_0 & \omega_1 x_1 & \ldots& \omega_n x_n\\
		g_0 & g_1 & \ldots &g_n
	\end{pmatrix}.$$

	For $1 \le r \le n-1,$ let $T_1, \ldots, T_r$ be $r$ tensors associated to $r$ weighted homogeneous derivations of degree $d_1, \ldots, d_r$ respectively. 
	Then $T_i = (g_0^i, g_1^i, \ldots, g_n^i) \in \bigoplus_{j = 0}^{n}(Sym^{d_{i,j}}k^{n+1}) \subset (Sym \, k^{n+1})^{\oplus {n+1}},$ for $i=1, \ldots,r$ and $d_{i,j} := d_i - (1 - \omega_j)$. 
	The \textit{weighted Multiple eigenscheme} (wME-scheme)
	of $T_1, \ldots, T_r$ is the closed subscheme $E(T_1,\ldots, T_r) \subset \mathbb{P}^n_\omega$ defined by the $(r+1) \times (r+1)$ minors of the homogeneous polynomial matrix
	
	$$M_{T_1, \ldots, T_{r}} = \begin{pmatrix}
		\omega_0 x_0 & \omega_1 x_1 & \ldots& \omega_n x_n\\
		g_0^1 & g_1^1 & \ldots &g_n^1\\
		\vdots &&&\vdots\\
		g_0^r & g_1^r & \ldots &g_n^r
	\end{pmatrix}.$$
\end{definition}

\begin{remark}\label{rmkdeterminantalSaitoScheme}
	What we observed is that the schemes of interest in the interpretations of Saito's criterion \cite[Theorem $2.5$]{JV_Eigenscheme2024} and \cite[Proposition $3.3$]{digennaro2025saitostheoremrevisitedapplication} are 
	determinantal schemes of a more general type:
	consider a divisor in $\mathbb{P}^n$ defined by a homogeneous polynomial $f$ of degree $d$ and
	let $\{ \delta_1, \ldots, \delta_n\}$ be a $S$-linearly independent set of $n$ homogeneous derivations of $\rom{Der}(f)$ over $\kk[x_0, \ldots, x_n]$ and form the matrix of partially symmetric tensors corresponding to
	$\delta_1, \ldots, \delta_n$,
	$$M = \begin{pmatrix}
		g_0^1 & g_1^1 & \ldots &g_n^1\\
		\vdots &&&\vdots\\
		g_0^n & g_1^n & \ldots &g_n^n
	\end{pmatrix}.$$
	
	Then by Saito's criterion, the divisor is free if and only if $f \in I_M$, where $I_M$ is the ideal of maximal minors of $M$, since $f$ can be written as the determinant of a Saito matrix,
	 and the exponents of the free divisor are
	\begin{enumerate}
		\item $(d_1, \ldots, d_{n-1}, d - \sum_{i=1}^{n-1} d_i)$ if the Euler derivation $\delta_{\mathcal{E}} = c \delta_i$ for some $i \in \{1, \ldots, n\}$ and $c \in \kk^*$,
		\item $(d_1, \ldots, d_n)$ otherwise.
	\end{enumerate}
	If the Euler derivation is part of the independent set of $n$ derivations modulo a scalar, then we recover the eigenscheme or Multiple eigenscheme with the free exponents of (i), otherwise we 
	consider the more general determinantal scheme (note, it is not an eigenscheme/Multiple eigenscheme) with the same statement of \cite[Proposition $3.3$]{digennaro2025saitostheoremrevisitedapplication} but with 
	free exponents given in (ii). The proof is easy enough since one just need to add the Euler derivation. The main difficulty is
	in the eigenscheme or Multiple eigenscheme case.
\end{remark}

Based on Remark \ref{rmkdeterminantalSaitoScheme}, we see that Definition \ref{defwME-scheme} is the right one to consider in weighted projective space to prove a weighted version of 
\cite[Proposition $3.3$]{digennaro2025saitostheoremrevisitedapplication}. However, we want to
point out that the determinantal schemes defined by the maximal minors of the homogeneous polynomial matrices $M_{T}$ and $M_{T_1, \ldots, T_{r}}$ in Definition \ref{defwME-scheme} are in fact not eigenschemes in 
$\mathbb{P}^n_\omega$: the first row of the matrices corresponds to the weighted Euler derivation and a point $[x_0: x_1: \ldots: x_n]_\omega \in \mathbb{P}^n_\omega$ is not equal to 
$[\omega_0 x_0: \omega_1 x_1: \ldots: \omega_n x_n]_\omega$ since there are in general no $t \in \kk$ such that 
$[\omega_0 x_0: \omega_1 x_1: \ldots: \omega_n x_n]_\omega  = [t^{\omega_0} x_0: t^{\omega_1 x_1}: \ldots: t^{\omega_n} x_n]_\omega$. Therefore, one does not have the set theoretic interpretation via the rational 
map $\mathbb{P}^n_\omega \cdots\to \mathbb{P}^n_\omega, p \mapsto [g_0(p): \ldots : g_n(p)]_\omega$, since it is not an eigenscheme. Instead, the weighted Multiple eigenscheme can be described set theoretically as 
the closure of the union of points $[x_0: x_1: \ldots: x_n]_\omega$ such that the corresponding points $[\omega_0 x_0: \omega_1 x_1: \ldots: \omega_n x_n]_\omega$ are in the linear space spanned 
by  $[g^i_0(p): \ldots: g^i_n(p)]_\omega$ for $1 \le i \le r$. In the expected codimensions, the weighted eigenschemes and weighted Multiple eigenschemes of Definition \ref{defwME-scheme} 
and even the more general determinantal schemes in Remark \ref{rmkdeterminantalSaitoScheme} in weighted projective space have coordinate rings $S/I_{E(T)}, S/I_{E(T_1,\ldots, T_r)}$ and $S/I_{M}$ that are 
Cohen-Macaulay (see \cite[Theorem $2.7$]{DetRings_Bruns88}, \cite[Theorem $2.1.5$]{Bruns_Herzog_1998} or \cite{HochsterEagon1971}), their defining ideals saturated and are also standard determinantal schemes.\\

The following technical lemma shows how to compute the weighted degree of the determinant of a matrix of tensors corresponding to weighted homogeneous derivations; in particular,
when the matrix is a Saito matrix for a free divisor $V(f) \subset \mathbb{P}^n_\omega$, then it shows how $\rom{deg}_\omega(f)$ can be written in terms of $d_0, d_1, \ldots, d_n$, the degrees of derivations 
of a basis for $\rom{Der}(f)$. 

\begin{lemma} \label{lemwDegDetM}
	Let $\delta_0, \ldots, \delta_n$ be weighted homogeneous derivations of degree $d_0, \ldots, d_{n}$ and denote by
	 $T_i = ( g_0^i, g_1^i,\ldots ,g_n^i) \in \bigoplus_{j = 0}^{n}(Sym^{d_{i,j}}k^{n+1}) \subset (Sym \, k^{n+1})^{\oplus {n+1}},$ for $i=0, \ldots,n$ and $d_{i,j} = d_i - (1 - \omega_j)$, 
	 their associated tensors. Then the determinant of the matrix of tensors
	$$M= \begin{pmatrix}
		g_0^0 & g_1^0 & \ldots &g_n^0\\
		\vdots &&&\vdots\\
		g_0^n & g_1^n & \ldots &g_n^n
	\end{pmatrix} $$
	is a weighted homogeneous polynomial of $\rom{deg}_\omega ( \rom{det}(M) ) = \sum_{i = 0}^n d_i + \sum_{i = 0}^n(\omega_i - 1).$
\end{lemma}

\begin{proof}
	$$\rom{det}(M) = \sum_{\sigma \in S_n} \rom{sgn}(\sigma)  g_{\sigma(0)}^0 \cdots g_{\sigma(n)}^n,$$
	where $S_n$ denotes the symmetric group in $n$ elements, and $\rom{sgn}(\sigma)$, the signature of a permutation $\sigma.$
	For any permutations $\sigma \in S_n$, one has that
	\begin{align*}
		\rom{deg}_\omega ( g_{\sigma(0)}^0 \cdots g_{\sigma(n)}^n ) &= \sum_{i=0}^n \rom{deg}_\omega (g_{\sigma(i)}^i)\\
		&= \sum_{i=0}^n \left( d_i - \rom{deg}_\omega (\partial_{x_{\sigma(i)}}) \right) \\
		&= \sum_{i=0}^n \left( d_i - (1- \omega_{\sigma(i)}) \right) \\
		&= \sum_{i = 0}^n d_i + \sum_{i = 0}^n(\omega_i - 1).
	\end{align*}
	The second equality follows from the fact that the derivations are weighted homogeneous
	and the third equality is by Definition \ref{defGradingDerS}. This shows that $ \rom{det}(M)$ is a weighted homogeneous
	polynomial of degree $\sum_{i = 0}^n d_i + \sum_{i = 0}^n(\omega_i - 1).$
\end{proof}

\begin{remark} \label{remfreeDetSumwDeg}
	In particular, for a free divisor $V(f) \subset \mathbb{P}^n_\omega$ with basis $\delta_0, \ldots, \delta_n$ of weighted degree $d_0, \ldots, d_{n}$, then
	$$\rom{deg}_\omega(f) = \sum_{i = 0}^n d_i + \sum_{i = 0}^n(\omega_i - 1)$$ by Saito's criterion.
\end{remark}

We are now able to state the main theorem relating weighted Multiple eigenschemes to freness of divisors in weighted projective spaces.

\begin{theorem} \label{propMEscheme}
	Let $V(f) \subset \mathbb{P}^n_\omega$ be a reduced divisor of weighted degree $d$. Let $\delta_i = g_0^i \partial_{x_0} + g_1^i \partial_{x_1} + \ldots + g_n^i \partial_{x_n}$, for 
	$i = 1, \ldots, n-1$, be $n-1$ generators of $\rom{Der}(f)$ of weighted degree $d_1 \le \ldots \le d_{n-1}$ with respective tensors 
	$T_i = ( g_0^i, g_1^i,\ldots ,g_n^i) \in  (Sym \, \kk^{n+1})^{\oplus {n+1}}$ such that $E(T_1, \ldots, T_{n-1})$ is in codimension $2$.
	Then $V(f)$ is free with weighted exponents $$(d_1, \ldots ,d_{n-1}, d + \sum_{i=0}^n (1-\omega_i) - d_1 - \ldots - d_{n-1} - 1)$$ if and only if $f \in I_{E(T_1, \ldots, T_{n-1})}$ is in the ideal of the 
	weighted Multiple eigenscheme of the tensors.
\end{theorem}

\begin{proof}
	Suppose $V(f)$ is free with weighted exponents $(d_1, \ldots ,d_{n-1}, d + \sum_{i=0}^n (1-\omega_i) - d_1 - \ldots - d_{n-1} - 1).$ Since 
	$E(T_1, \ldots, T_{n-1})$ has codimension $2$, $\delta_i$ for $i \in \{1, \ldots, n-1\}$ together with weighted Euler derivation 
	$\delta_{\mathcal{E}_\omega}$ form part of a basis for $\rom{Der}(f)$, and since $V(f)$ is free with the given exponents, then there exists a derivation $\mu$ of
	degree $d + \sum_{i=0}^n (1-\omega_i) - d_1 - \ldots - d_{n-1} - 1$ such that $\delta_{\mathcal{E}_\omega}, \delta_1, \ldots ,\delta_{n-1}, \mu$ is a basis
	for $\rom{Der}(f)$. By Saito's criterion, the determinant of its Saito matrix $$\rom{det}[\delta_{\mathcal{E}_\omega}, \delta_1, \ldots ,\delta_{n-1}, \mu]= uf,$$ where $u$ is a unit and w.l.o.g. one can take $u = 1$.
	Note that the Saito matrix with the row corresponding to the derivation $\mu$ deleted, is exactly the matrix $M_{T_1, \ldots, T_{n-1}}$ of the wME-scheme, whose maximal
	minors are the generators of $I_{E(T_1, \ldots, T_{n-1})}$. By calculating the determinant of the Saito matrix by expanding along the row of $\mu$ implies $f \in I_{E(T_1, \ldots, T_{n-1})}$.
	
	Conversely, if $f \in I_{E(T_1, \ldots, T_{n-1})}$, then there exists a tensor $T_n = ( g_0^n, g_1^n,\ldots ,g_n^n)$ such that the corresponding derivation
	$\mu = \sum_{i = 0}^n g_i^n \partial_{x_i}$ is weighted homogeneous. When $T_n$ is added to the ME-scheme matrix $M_{T_1, \ldots, T_{n-1}}$, 
	computing the determinant of the resulting matrix gives $f$. 
	Next we show that $\mu \in \rom{Der}(f)$, which by Saito's criterion would imply that $\rom{Der}(f)$ is free with 
	$\delta_{\mathcal{E}_\omega}$, $\delta_i$ for $i \in \{1, \ldots, n-1\}$ and $\mu$ as basis. 
	One finds that because of our assumption on the codimension of $E(T_1, \ldots, T_{n-1})$ and Lemma \ref{lemGcdCodim} below that a similar argument as in the proof of 
	\cite[Proposition $3.3$]{digennaro2025saitostheoremrevisitedapplication} and \cite[Theorem $2.5$]{JV_Eigenscheme2024} works: 
	first, add the tensor $T_n$ corresponding to derivation $\mu$ as a new row of the matrix $M_{T_1, \ldots, T_{n-1}}$ and denote the resulting matrix by $M$. Denote also $CoM^T$, the transpose of the matrix of 
	cofactors of $M$. Then $$CoM^T \ M = fI,$$ and multiplying by $\nabla(f)$, the column vector of partial
	derivatives of $f$, we have 
	$$CoM^T \ M \nabla(f) = CoM^T [df, fK_1, \ldots, fK_{n-1}, \mu(f)]^T,$$ where $K_i \in \kk[x_0, \ldots, x_n] =:S,$ 
	and obtain a system of equations $$\{m_{Q_i}\mu(f) = f h_i\}_{i =0, \ldots, n},$$ where the $m_{Q_i}$'s are the generators of 
	$I_{E(T_1, \ldots, T_{n-1})}$ and $h_i \in S$ for $i =0, \ldots, n$. Since $E(T_1, \ldots, T_{n-1})$ has codimension $2$, by Lemma \ref{lemGcdCodim} below, 
	$$\rom{gcd}(m_{Q_1}, \ldots, m_{Q_{n+1}}) = 1,$$ and this implies that $f$ does not divide all the generators $m_{Q_i}$ of $I_{E(T_1, \ldots, T_{n-1})}$, hence
	$f \,|\, \mu(f)$. Therefore $\rom{Der}(f)$ is free and by Remark \ref{remExistsBasisHomog}, there exists a basis of weighted homogeneous derivations of degree $d_0, d_1, \ldots, d_{n-1}, d_n$, where 
	$d_0 = \rom{deg}_\omega(\delta_{\mathcal{E}_\omega}) = 1$, $d_1, \ldots, d_{n-1}$ as given, and $d_n := \rom{deg}_\omega(\mu).$ By Lemma \ref{lemwDegDetM} and Remark \ref{remfreeDetSumwDeg}, 
	$$d =\rom{deg}_\omega(f) = \sum_{i = 0}^n d_i + \sum_{i = 0}^n(\omega_i - 1),$$ and hence
	$$d_n = d + \sum_{i=0}^n (1-\omega_i) - (\sum_{i = 1}^{n-1} d_i) - 1.$$
\end{proof}	

An easy criteria to check if the weighted ME-scheme $E(T_1, \ldots, T_{n-1}) \subset \mathbb{P}^n_\omega$ of $n-1$ tensors is in codimension $2$ is the following lemma.

\begin{lemma} \label{lemGcdCodim}
	With the above notations, let $I_{E(T_1, \ldots, T_{n-1})} = (m_{Q_1}, \ldots, m_{Q_{n+1}})$ be the ideal of $n \times n$ minors of $M_{T_1, \ldots, T_{n-1}}$, the matrix of weighted ME-scheme. 
	Then $\rom{gcd}(m_{Q_1}, \ldots, m_{Q_{n+1}}) = 1$ if and only if $\rom{codim}(I_{E(T_1, \ldots, T_{n-1})}) = 2.$
\end{lemma}

\begin{proof}
	Let $S = \kk[x_0, \ldots, x_n].$ By the Giambelli–Thom–Porteous formula (see \cite[Ch $12.1$]{eisenbud20163264}), the expected codimension 
	$\rom{codim}(I_{E(T_1, \ldots, T_{n-1})}) = 2$, else $\rom{codim}(I_{E(T_1, \ldots, T_{n-1})}) \le 2$. We have 
	$\rom{codim}(I_{E(T_1, \ldots, T_{n-1})})  = 0$ if and only if $I_{E(T_1, \ldots, T_{n-1})} = (0)$ i.e. $\rom{gcd}(m_{Q_1}, \ldots, m_{Q_{n+1}}) = 0$, and
	$\rom{codim}(I_{E(T_1, \ldots, T_{n-1})})  = 1$ if and only if $I_{E(T_1, \ldots, T_{n-1})} = (hm'_{Q_1}, \ldots, hm'_{Q_{n+1}})$ for some non-constant polynomial $h$,
	i.e. $\rom{gcd}(m_{Q_1}, \ldots, m_{Q_{n+1}}) =h$.
	Hence $\rom{codim}(I_{E(T_1, \ldots, T_{n-1})})  = 2$ if and only if $\rom{gcd}(m_{Q_1}, \ldots, m_{Q_{n+1}}) = 1$.
\end{proof}

\begin{example} \label{exWWhitney}
	Consider $f = x^a - y^b z^c \in \kk[x,y,z]$ for positive integers $a,b,c$. Observe that $f$ is weighted homogeneous for all possible choices 
	of positive integers $a,b,c$: 
	take weight vector 
	$\omega = (2bc, ac, ab)$ for instance, then $\rom{deg}_{\omega}(f) = 2abc$. Hence one can view $f$ as the defining polynomial of a divisor in 
	$\mathbb{P}^2_\omega$. By considering the weighted eigenscheme of the tensor $T_1 = (bx, ay, 0)$ associated to derivation 
	$\delta_1 = bx \partial_x + ay \partial_y$, which is in $\rom{Der}(f)$, one has that the ideal of maximal minors of
	$$M_{T_1} = \begin{pmatrix}
		2bc x & acy & abz\\
		bx & ay & 0
	\end{pmatrix} \ \text{is}\ I_{E(T_1)} = (yz, xz, xy), \ \text{which defines a finite scheme}$$
	$$E(T_1) = \{[1:0:0]_\omega, [0:1:0]_\omega, [0:0:1]_\omega \} \subset \mathbb{P}^2_\omega.$$
	Since $f \notin I_{E(T_1)},$ by Theorem \ref{propMEscheme} the divisor $V(f)$ is not free with weighted exponents $(1, 2abc - 2bc -ac-ab+1)$,
	and in fact $V(f)$ is not free.
	
	However, $zf \in I_{E(T_1)}$ and one checks easily that $\delta_1 \in \rom{Der}(zf)$, so Theorem \ref{propMEscheme} says that 
	$$\text{divisor} \ V(zf)\  \text{is free with weighted exponents}\ (1, 2abc - 2bc -ac+1).$$ 
	As a scheme, $V(zf) \supset E(T_1)$ contains the weighted eigenscheme of $T_1$ if
	and only if the principal ideal $(zf) \subset I_{E(T_1)}$. In the case of the divisor $V(f)$, the point $[1:0:0]_\omega \notin V(f).$
	
	A parallel situation is to instead consider the weighted eigenscheme of the tensor $T_2 = (cx, 0, az)$ associated to derivation 
	$\delta_2 = cx \partial_x + az \partial_z$. The defining ideal is also $I_{E(T_2)} = (yz, xz, xy),$ so it defines the same scheme as $E(T_1).$
	Moreover, observe that $yf \in I_{E(T_2)}$ and that $\delta_2 \in \rom{Der}(yf)$, so one can apply Theorem \ref{propMEscheme}, which shows that 
	$$\text{divisor} \ V(yf )\  \text{is free with weighted exponents}\ (1, 2abc - 2bc -ab+1).$$
\end{example}

\begin{remark}
	Consider $V(f)$ now as a divisor in affine space $\mathbb{A}^3_{\kk}$. When $a = b = 2, c=1$, then $V(f)$ is the Whitney umbrella, which can be seen
	over the reals, $\kk = \mathbb{R}$ as a ruled surface, self-intersecting along a half-line and a pinch at the origin. The Whitney umbrella is not a free divisor,
	but ajoin it with a plane $z=0$ or $y=0$, then the resulting divisor is free as shown implicitly in Example \ref{exWWhitney}; the next section discusses 
	the viewpoint of weighted projective space and affine space in regard to freeness. When $a = 2, b = c = 1$, then $V(yf)$ is the quadratic cone ajoint a plane,
	a free divisor (c.f. \cite[Example $3.7$]{Buchweitz_Conca_2012} and see \cite{mond-schulze2013}).
\end{remark}

\section{From free divisors in $\mathbb{P}_\omega^{n-1}$ to free divisors in $\mathbb{A}^{n}$ and the Cone Construction for $\mathbb{P}^{n}$} 
\label{secwFreeDiv}

In the case that one starts with a free divisor $D$ in $\mathbb{P}_\omega^{n-1}$, defined by a weighted homogeneous polynomial $F$ over $R := \kk[x_1, \ldots, x_{n}]$, then the $\mathbb{Z}$-graded
$R$-module $\rom{Der}(F) := \rom{Der}_{\kk}(-\rom{log}\, D)$ is a free $R$-module by definition. Hence forgetting the grading on the $R$-module, $\rom{Der}(F)$ is still a free $R$-module. One can
now just view $F$ as a polynomial in the affine coordinate ring $\kk[x_1, \ldots, x_{n}]$ and since $\rom{Der}(F)$ is free, the divisor in $\mathbb{A}^{n}$ defined by $F$ is free. We state this fact in the
 following lemma.
 
\begin{lemma} \label{lemFreeWeightedEqualFreeAffine}
	Given a free divisor in weighted projective $(n-1)$-space defined by the polynomial $F$, the corresponding divisor in affine $n$-space defined by $F$ is also free.
\end{lemma}

\begin{remark}
	One can pay particular attention to the homogeneous case of free divisors in $\mathbb{P}_\omega^{n-1}$, where the weights are all the same. These are free divisors in standard
	projective space $\mathbb{P}^{n-1}$.
\end{remark}

\subsection{Cone Construction for $\mathbb{P}^{n}$} \label{secConeConstruct}

From a free divisor in weighted projective space $\mathbb{P}_\omega^{n-1}$, one gets (i) a corresponding free 
divisor in affine space $\mathbb{A}^{n}$ and (ii) can construct a divisor in standard projective space $\mathbb{P}^{n}$ 
by the \textit{cone construction}, which is similar to taking the cone over a hyperplane arrangement (see \cite[Definition $1.15$]{orlik_terao_1992arrangements}), and was also mentioned in \cite{st2014}
The construction is as follows:

Projective $n$-space $\mathbb{P}^{n}$ has coordinate ring $\kk[x_0, \ldots, x_{n}]$, so $F$ as above,
\begin{enumerate}
	\item  we homogenize $F$ with respect to the new variable $x_0$ and denote it by $F^h$,
	\item then consider the divisor in $\mathbb{P}^{n}$ defined by the homogeneous polynomial $x_0 F^h$.
\end{enumerate}

\subsection{Homogenization and Freeness in projective space}

As just demonstrated, we denote by $\_^h$  the homogenization w.r.t. the variable $x_0$ in $\kk[x_0, ... , x_n]$, which when applied to a derivation $\delta= \sum_i  P_i \partial_{x_i}$, 
by abuse of notation we mean that $\delta^h$ is a homogeneous derivation, i.e. $\rom{deg}\, P_i = d$ for all $i$ and some fixed $d \in \mathbb{N}.$ Dehomogenizing a homogeneous polynomial $f$ or 
derivation $\delta$ in $x_0$, i.e. $f_{|_{x_0 = 1}} := f(1,x_1, \ldots, x_n)$ will be denoted $f_{|_{x_0 = 1}}$ and $\delta_{|_{x_0 = 1}}$ respectively.

\begin{lemma} \label{lemHomogPolysDehomIsSame}
	Let $F_1, F_2 \in \kk[x_0, x_1, \ldots, x_n]$ be homogeneous polynomials of degre $d_1$ and $d_2$ respectively, with $d_1 \ge d_2$.
	If $F_{1_{|_{x_0 = 1}}} = F_{2_{|_{x_0 = 1}}}$, then $F_1 = x_0^{d_1 - d_2} F_2.$
\end{lemma}

\begin{proof}
	Suppose $F_{1_{|_{x_0 = 1}}} = F_{2_{|_{x_0 = 1}}}$, hence it follows that 
	$(F_{1_{|_{x_0 = 1}}})^h = (F_{2_{|_{x_0 = 1}}})^h.$
	 Denote by $d_3 := \rom{deg}((F_{1_{|_{x_0 = 1}}})^h)
	= \rom{deg}((F_{2_{|_{x_0 = 1}}})^h).$ Since $F_1$ and $F_2$ are homogeneous, and 
	$d_3 =\rom{deg}(F_{2_{|_{x_0 = 1}}})$, then $d_3$ is at most $d_2$ and 
	$d_3 \le d_2 \le d_1$.
	
	Since $((F_{2_{|_{x_0 = 1}}})^h)_{|_{x_0 = 1}} = F_{2_{|_{x_0 = 1}}}$, this implies that
	\begin{align} \label{eqF2Isx0F2dehomHomog}
		F_2 = x_0^{d_2 - d_3} (F_{2_{|_{x_0 = 1}}})^h = x_0^{d_2 - d_3} (F_{1_{|_{x_0 = 1}}})^h.
	\end{align}
	Similarly, since $d_3 \le d_1$ and $((F_{1_{|_{x_0 = 1}}})^h)_{|_{x_0 = 1}} = F_{1_{|_{x_0 = 1}}}$, this implies that
	\begin{align*}
		F_1 &= x_0^{d_1 - d_3} (F_{1_{|_{x_0 = 1}}})^h \\
				&= x_0^{d_1 - d_2} x_0^{d_2 - d_3} (F_{1_{|_{x_0 = 1}}})^h\\
				&= x_0^{d_1 - d_2} F_2,
	\end{align*}
	where the last equality follows from (\ref{eqF2Isx0F2dehomHomog}).
	
\end{proof}

\begin{corollary} \label{corx_0MultipleOfDerivHomog}
	Let $f \in \kk[x_1, \ldots, x_n] =: R$ and $\delta \in \rom{Der}_{\kk}(R)$. Denote by $d_1 := \rom{deg}(\delta^h(f^h))$ and 
	$d_2 := \rom{deg}(\delta(f)^h)$. Then $\delta^h(f^h) = x_0^{d_1 - d_2}\delta(f)^h,$
	where $\_^h$ is homogenization w.r.t. the variable $x_0$ in $\kk[x_0, ... , x_n]$ as defined above.
\end{corollary}

\begin{proof}
	First note that $f \in \kk[x_1, \ldots, x_n]$ and $\delta \in \rom{Der}_{\kk}(R)$ so they do not contain the variable $x_0$ and in $\delta$, the coefficient of $\partial_{x_0}$ is zero. Therefore
	\begin{align}
		\delta(f)^h_{|_{x_0 = 1}} &= \delta(f), \label{eqDehomDelta_f}\\
		\delta^h(f^h)_{|_{x_0 = 1}} &= \delta^h_{|_{x_0 = 1}}(f^h_{|_{x_0 = 1}}) = \delta(f),
	\end{align} and hence 
	\begin{equation}\label{eqDehomDeltah_fhIsDelta_f_h}
		\delta^h(f^h)_{|_{x_0 = 1}}  = \delta(f)^h_{|_{x_0 = 1}}. 
	\end{equation}
	
	Now we claim that $\rom{deg}(\delta^h(f^h)) \ge \rom{deg}(\delta(f)^h)$: Since $\delta^h(f^h)$ is homogeneous, then by 
	(\ref{eqDehomDeltah_fhIsDelta_f_h}), 
	$$\rom{deg}((\delta(f)^h_{|_{x_0 = 1}})^h) = \rom{deg}((\delta^h(f^h)_{|_{x_0 = 1}})^h) \le \rom{deg}(\delta^h(f^h)).$$
	But by (\ref{eqDehomDelta_f}), we have
	$$\rom{deg}((\delta(f)^h_{|_{x_0 = 1}})^h) = \rom{deg}(\delta(f)^h) \le \rom{deg}(\delta^h(f^h)).$$
	
	Hence $\rom{deg}(\delta^h(f^h)) \ge \rom{deg}(\delta(f)^h)$ and $\delta^h(f^h)_{|_{x_0 = 1}}  = \delta(f)^h_{|_{x_0 = 1}}$,
	which by Lemma \ref{lemHomogPolysDehomIsSame} implies that $\delta^h(f^h) = x_0^{d_1 - d_2}\delta(f)^h.$
	 
\end{proof}

\begin{corollary} \label{corHomogenizingDerv}
	Let $f \in \kk[x_1, \ldots, x_n] =: R$ and $\delta \in \rom{Der}_{\kk}(R)$. Then $\delta  \in \rom{Der}(f)\subset \rom{Der}_{\kk}(R)$ if and only if 
	$\delta^h  \in \rom{Der}(f^h)\subset \rom{Der}_{\kk}(S),$ where $\_^h$ 
	is homogenization w.r.t. the variable $x_0$ in $S := \kk[x_0, ... , x_n]$ as defined above.
\end{corollary}

\begin{proof}
	$(\Leftarrow)$ Suppose $\delta^h\in \rom{Der}(f^h)$, then $\delta^h(f^h) = p f^h$, for some homogeneous polynomial $p$, and dehomogenizing, we
	have $\delta(f) = p(1,x_1, \cdots, x_n) f$, where $p(1,x_1, \cdots, x_n) \in \kk[x_1, \ldots, x_n]$, hence $\delta  \in \rom{Der}(f)$.
	
	$(\Rightarrow)$ Conversely, if $\delta \in \rom{Der}(f)$, then $\delta(f) = qf$, for some $q \in \kk[x_1, \ldots, x_n]$.
	Denote by $d_1 := \rom{deg}(\delta^h(f^h))$ and 
	$d_2 := \rom{deg}(\delta(f)^h)$. By Corollary \ref{corx_0MultipleOfDerivHomog}, 
	$\delta^h(f^h) =  x_0^{d_1 - d_2} \delta(f)^h = x_0^{d_1 - d_2}(qf)^h = x_0^{d_1 - d_2}q^h f^h$. 
\end{proof}

A direct consequence of Corollary \ref{corHomogenizingDerv} and Saito's criterion gives an easy way to check if the cone construction of a free affine divisor given by $f \in \kk[x_1, \ldots, x_n]$ is 
free in projective space, i.e. if the divisor defined by $x_0 f^h$ is also free.

\begin{prop} \label{propConeConstructionFree}
	Let $f \in R = \kk[x_1, \ldots, x_n]$ be a reduced polynomial. If the divisor in $\mathbb{A}^n$ defined by $f$ is free and there exists a basis $\{ \delta_1, \ldots, \delta_n \}$ of $\rom{Der}(f)$  
	such that $\sum^n_{i= 1} \rom{deg} (\delta_i) = \rom{deg}(f)$, then the divisor in $\mathbb{P}^{n}$ defined by 
	$x_0 f^h \in S:= \kk[x_0, \ldots, x_n]$ is free with exponents $(\rom{deg} (\delta_1), \ldots, \rom{deg} (\delta_n))$.
\end{prop}

\begin{proof}
	The basis $\{ \delta_1, \ldots, \delta_n \}$ for the free $R$-module $\rom{Der}(f)$ can be lifted to a set of derivations 
	$\{ \delta_1^h, \ldots, \delta_n^h \}$ of the $S$-module $\rom{Der}(f^h)$ by Corollary \ref{corHomogenizingDerv}, but since these 
	homogenized derivations have no term in $\partial_{x_0}$, they are also derivations of $\rom{Der}(x_0 f^h)$. Homogenizing the derivations
	in the new indeterminate $x_0$ does not change the linear independence of the set over $S$, nor does adding the Euler derivation
	 $\delta_\mathcal{E}$. Hence to show that $\{ \delta_\mathcal{E},\delta_1^h, \ldots, \delta_n^h \}$ is a basis for $\rom{Der}(x_0 f^h)$, by
	 Saito's criterion it suffices to check the degree of the determinant of Saito's matrix, i.e. the sum of degrees of the homogeneous derivations:
	 $\rom{deg} (\delta_\mathcal{E})  +\sum^n_{i= 1} \rom{deg} (\delta_i^h) =1+\sum^n_{i= 1} \rom{deg} (\delta_i) = 1+ \rom{deg}(f) = \rom{deg}(x_0 f^h)$.
\end{proof}

\section{Applications: New free divisors in affine and projective spaces} \label{secApplis}

\subsection{New free divisors coming from complete reflection arrangements}

In this section, we will denote the coordinate ring of $\mathbb{P}^{\ell-1}_\omega$ and $\mathbb{A}^\ell$ by $R := \kk[x_1, \ldots, x_{\ell}]$,
and $S := \kk[x_0, x_1, \ldots, x_{\ell}]$ will denote that of $\mathbb{P}^\ell$. Let  $\prod_{j \ne i} n_j := \prod_{j \ne i,\\ 1 \le j \le \ell} n_j = n_1 n_2 \ldots \widehat{n_i} \ldots n_{\ell-1} n_\ell$ 
for a fixed $i \in \{1, 
\ldots, \ell \}$.

Complete reflection arrangements $Q = x_1 \ldots x_\ell \prod_{1 \le i < j \le \ell} (x_i^n - x_j^n)$ as studied in \cite[Proposition $6.77$]{orlik_terao_1992arrangements} or
 \cite[Section $5$]{JV_Eigenscheme2024} are free divisors for all $n \ge 1$.
A nonstandard $\mathbb{Z}$-grading on the ring $S$ allows one to consider the corresponding weighted homogeneous polynomial
$$F = x_1 \ldots x_\ell \prod_{1 \le i < j \le \ell} (x_i^{n_i} - x_j^{n_j})$$ 
for arbitrary choices of positive integers $n_i$. The polynomial $F$ corresponds to the defining equation of a 
hypersurface in $\mathbb{P}^{\ell-1}_\omega$ for some weight vector $\omega = (\omega_1, \ldots, \omega_\ell)$, and 
$F$ weighted homogeneous implies that $n_i\omega_i = n_j\omega_j$ for all $i,j \in \{1, \ldots, \ell\}$. 
Observe that one possible choice of weight vector is $\omega = (\omega_1, \ldots, \omega_\ell)$ with $\omega_i = \prod_{j \ne i} n_j$. 

One will show that these define free divisors in 
weighted projective space and one can give the free weighted exponents. Moreover, corresponding divisors in affine spaces and by the cone construction, in projective spaces, will also be free, 
and we state the free exponents in the later case.

\begin{lemma} \label{lem_vanDerMonde_der}
	Let $F = x_1 \cdots x_\ell \prod_{1 \le i < j \le \ell}(x_i^{n_i} - x_j^{n_j})$ for any positive integers $n_q \in \mathbb{Z}_{>0}$ and let
	$\omega = (\omega_1, \ldots, \omega_\ell)$ be a weight vector such that $F$ is weighted homogeneous. Then for 
	$0 \le m \le \ell-1$, $\sum_{i = 1}^\ell \omega_i x_i^{mn_i + 1}\partial_{x_i} \in \rom{Der}(F)$, and in particular, 
	$\sum_{i = 1}^\ell(\prod_{j \ne i}n_j) x_i^{mn_i + 1}\partial_{x_i} \in \rom{Der}(F)$.
\end{lemma}

\begin{proof}
	Let $\mu := \sum_{i = 1}^\ell \omega_i  x_i^{mn_i + 1}\partial_{x_i}$ for some fixed $0 \le m \le \ell-1$. Then 
	$\mu(x_i)= \omega_i  x_i^{mn_i+1}$ implies $\mu \in \rom{Der}(x_i)$. In addition, for $\mu(x_i^{n_i} - x_j^{n_j})$ for $i < j$, 
	we have $$\mu(x_i^{n_i} - x_j^{n_j}) = n_i \omega_i ((x_i^{n_i})^{m+1} - (x_j^{n_j})^{m+1} )$$ since $n_i \omega_i = n_j \omega_j$ for all $i,j \in \{1, \ldots, \ell\}$ if 
	$F$ is weighted homogeneous. 
	Now note that $x_i^{n_i} - x_j^{n_j}$ divides $ ((x_i^{n_i})^{m+1} - (x_j^{n_j})^{m+1} )$ since $(\alpha - 1) \, | \, (\alpha^{m+1} - 1) $ as $\alpha = 1 $ is a 
	$(m+1)$-root of unity, where $\alpha := x_i^{n_i}/x_j^{n_j}$; hence $\mu \in \rom{Der}(x_i^{n_i} - x_j^{n_j})$. Hence by (\ref{eqDerIntersection}),
	 $$ \mu \in \bigcap_i \rom{Der}(x_i) 
	  \cap \bigcap_{1 \le i < j \le \ell} \rom{Der}(x_i^{n_i} - x_j^{n_j})  = \rom{Der}(F).$$
	  Now it is clear that $F$ is weighted homogenous when $\omega_i = \prod_{j \ne i} n_j$ for $1 \le i \le \ell.$
\end{proof}

\begin{theorem} \label{thmGeneralizedW}
	Let $n_1, \ldots, n_\ell$ be any positive integers. The divisor $\mathcal{D}_\ell^\ell(\{n_i\})$ in weighted projective space 
	$\mathbb{P}^{\ell-1}_\omega$ defined by the weighted homogeneous polynomial
	$$F = x_1 \cdots x_\ell \prod_{1 \le i < j \le \ell}(x_i^{n_i} - x_j^{n_j})$$ is free with weighted exponents 
	$( n_i\omega_i + 1, 2n_i\omega_i + 1, \ldots, (\ell - 1)n_i\omega_i + 1)$ for any $i \in \{1, ... , \ell\}$.
\end{theorem}

\begin{proof}
	By Lemma \ref{lem_vanDerMonde_der}, for $0 \le m \le \ell-1$, the derivation $\delta_m = \sum_{i = 1}^\ell \omega_i x_i^{mn_i + 1}\partial_{x_i}$ is weighted homogeneous of 
	degree $m n_i \omega_i + 1$ for any $i \in \{1, ... , \ell\}$ and is in $\rom{Der}(F)$.
	Building Saito's matrix $M$ out of the above $\ell$ derivations and computing its determinant, one gets $\rom{det}\, M=  u x_1 \cdots x_\ell \, \rom{det}(V_{\ell-1})$ where 
	$u = \prod_{i = 1}^\ell \omega_i \ge 1$ is a unit and $V_{\ell-1}[x_\ell^{n_\ell}, x_{\ell-1}^{n_{\ell-1}}, \ldots, x_1^{n_1} ] = V_{\ell-1}$ is a
	 Vandermonde matrix in the $x_i^{n_i}$'s. Hence $\rom{det}\, M = u x_1 \cdots x_\ell \, \rom{det}(V_{\ell-1}) = u F$, which implies by 
	 Saito's criterion that $\rom{Der}(F)$ is free with basis $\{\delta_m \ | \ 0 \le m \le \ell-1 \}$.
\end{proof}

Complementing the study of divisors related but not necessarily coming from complete reflection arrangements, consider now 
$$F_k := p \cdot \prod_{1 \le i < j \le \ell}(x_i^{n_i} - x_j^{n_j}),\ \ \text{where} \ p = 
\begin{cases}
	1, & \text{if}\  k = 0\\
	x_1 \cdots x_k, & \text{if}\ 1 \le k \le \ell
\end{cases}$$
 for all $0 \le k \le \ell$. We will show an analogue of \cite[Proposition $6.85$]{orlik_terao_1992arrangements} proving that divisors defined by $F_k$ are also free. Inspired by the notations 
 $\mathcal{A}^k_\ell(n)$ of \cite[Chapter $6.4$]{orlik_terao_1992arrangements},  we similarly denote our divisors by $\mathcal{D}_\ell^k(\{n_i\})$ (or $\widetilde{\mathcal{D}}_\ell^k(\{n_i\})$ in 
 the standard projective space).

\begin{lemma} \label{lemDerivMuOfFk}
	Let $F_k = p \cdot \prod_{1 \le i < j \le \ell}(x_i^{n_i} - x_j^{n_j})$ for any positive integers $n_q \in \mathbb{Z}_{>0}$, where $p = 
	\begin{cases}
		1, & \text{if}\  k = 0\\
		x_1 \cdots x_k, & \text{if}\ 1 \le k \le \ell
	\end{cases}$ and let $\omega = (\omega_1, \ldots, \omega_\ell)$ be a weight vector such that $F_k$ is weighted homogeneous. Then
	$\mu = \sum_{i = 1}^\ell p \omega_i \cdot x_1^{n_1 - 1} \cdots \widehat{x_i^{n_i - 1}} \cdots x_\ell^{n_\ell - 1} \partial_{x_i} \in \rom{Der}(F_k),$
	and in particular $\mu = \sum_{i = 1}^\ell p \cdot  \prod_{j \ne i}n_j \cdot x_1^{n_1 - 1} \cdots \widehat{x_i^{n_i - 1}} \cdots x_\ell^{n_\ell - 1} \partial_{x_i} \in \rom{Der}(F_k).$
\end{lemma}

\begin{proof} Let $1 \le i < j \le \ell$, then
	\begin{align*}
		\mu(x_i^{n_i} - x_j^{n_j}) 
		&= \, p \omega_i \cdot x_1^{n_1 - 1} \cdots \widehat{x_i^{n_i - 1}} \cdots x_\ell^{n_\ell - 1} \cdot n_i x_i^{n_i - 1} 
		- \, p \omega_j \cdot x_1^{n_1 - 1} \cdots \widehat{x_j^{n_j - 1}} \cdots x_\ell^{n_\ell - 1} \cdot n_j x_j^{n_j - 1}\\ 
		&= p\cdot x_1^{n_1 - 1} \cdots  x_\ell^{n_\ell - 1}(n_i\omega_i - n_j\omega_j) = 0
	\end{align*}
	since $n_i\omega_i = n_j\omega_j$.
	This implies that $\mu \in \rom{Der}(x_i^{n_i} - x_j^{n_j})$. 
	
	If $k = 0$, then 
	$$\mu \in \bigcap_{1 \le i < j \le \ell} \rom{Der}(x_i^{n_i} - x_j^{n_j}) = \rom{Der}(F_k).$$
	Else, $1 \le k \le \ell$ and let $1 \le m \le k$, then $\mu(x_m) = p \omega_m \cdot x_1^{n_1 - 1} \cdots \widehat{x_m^{n_m - 1}} \cdots x_\ell^{n_\ell - 1}  \in \rom{Der}(x_m)$ 
	since by definition $x_m\, |\, p.$ Therefore we again have that 
	$$\mu \in \bigcap_{i = 1}^k \rom{Der}(x_i) \cap \bigcap_{1 \le i < j \le \ell}
	\rom{Der}(x_i^{n_i} - x_j^{n_j}) = \rom{Der}(F_k).$$
	
	As in the proof of Lemma \ref{lem_vanDerMonde_der}, it is clear that $F_k$ is weighted homogenous when $\omega_i = \prod_{j \ne i} n_j$ for $1 \le i \le \ell.$
\end{proof}

We denote the divisor $\mathcal{D}_\ell^k(\{n_i\}) := \mathcal{D}_\ell^\ell(\{n_i\})/ x_{k+1}x_{k+2} \cdots x_\ell$ for $0 \le k \le \ell$.

\begin{theorem} \label{thmW_Dkl}
	Let $n_1, \ldots, n_\ell$ be any positive integers. The divisor $\mathcal{D}_\ell^k(\{n_i\})$ in weighted projective space 
	$\mathbb{P}^{\ell-1}_\omega$ defined by the weighted homogeneous polynomial
	$$F_k = p \cdot \prod_{1 \le i < j \le \ell}(x_i^{n_i} - x_j^{n_j}),\ \ \text{where} \ p = 
	\begin{cases}
		1, & \text{if}\  k = 0\\
		x_1 \cdots x_k, & \text{if}\ 1 \le k \le \ell
	\end{cases}$$
	is free with weighted exponents 
	$$(n_i\omega_i + 1, 2n_i\omega_i + 1, \ldots,  (\ell-2)n_i\omega_i + 1, (\ell-1)n_i\omega_i + 1 - \sum_{j = k+1}^\ell \omega_j),$$
	 for any $i \in \{1, ... , \ell\}.$
\end{theorem}

\begin{proof}
	Let $R = \kk[x_1, \ldots, x_\ell]$. Since $F_{k} = p \cdot\prod_{1 \le i < j \le \ell}(x_i^{n_i} - x_j^{n_j})$, the $\ell-1$ $R$-linearly independent derivations in Lemma 
	\ref{lem_vanDerMonde_der} for $0 \le m \le \ell-2$ are elements of 
	\begin{align*}
		\bigcap_{i = 1}^\ell \rom{Der}(x_i) \cap \bigcap_{1 \le i < j \le \ell} \rom{Der}(x_i^{n_i} - x_j^{n_j})
		&\subset \bigcap_{i = 1}^k \rom{Der}(x_i) \cap 
		\bigcap_{1 \le i < j \le \ell} \rom{Der}(x_i^{n_i} - x_j^{n_j}) \\
		&= \rom{Der}(F_{k}),
	\end{align*}
	and hence are part of a system of generators of $\rom{Der}(F_{k})$. Form the matrix corresponding to the $\ell-1$ derivations and denote by $T'_i$ the tensors
	corresponding to the derivations in Lemma \ref{lem_vanDerMonde_der} for $1 \le m \le \ell-2$, which are the $2$nd to last row of the following matrix:
	\begin{equation} \label{matrixN1_l-2}
	N_{T'_1, \ldots, T'_{\ell-2}} = \begin{pmatrix}
		\omega_1 x_1 & \omega_2 x_2 & \ldots & \omega_\ell x_\ell\\
		\omega_1 x_1^{n_1+1} & \omega_2 x_2^{n_2+1} & \ldots & \omega_\ell x_\ell^{n_\ell+1}\\
		\vdots & \vdots & \ldots & \vdots \\
		\omega_1 x_1^{(\ell-2)n_1+1} & \omega_2 x_2^{(\ell-2)n_2+1} & \ldots & \omega_\ell x_\ell^{(\ell-2)n_\ell+1}
	\end{pmatrix}.
	\end{equation}
	
	Furthermore, consider the tensor 
	$T_{\ell-1} = (g_1^{\ell-1}, g_2^{\ell-1},\ldots ,g_\ell^{\ell-1} )$ corresponding to the derivation $\mu \in \rom{Der}(F_{k})$ in 
	Lemma \ref{lemDerivMuOfFk}, where 
	$$g_i^{\ell-1} =  p \omega_i \cdot  x_1^{n_1-1}\cdots\widehat{x_i^{n_i-1}}\cdots x_\ell^{n_\ell-1} \ \text{for}\ 1 \le i \le \ell;$$
	$T_{\ell-1}$ is added as a new row of the matrix $N_{T'_1, \ldots, T'_{\ell-2}} $.
	
	Until the end of the proof, we will consider matrices with coefficients in $\rom{Frac}(\kk[x_1, \ldots, x_\ell])$:
	\begin{align*}
		\rom{det} \, &\begin{pmatrix}
			N_{T'_1, \ldots, T'_{\ell-2}}\\ \hdashline[1.5pt/2pt]
			T_{\ell-1}
		\end{pmatrix} =\\
		 &p \omega_1 \cdots  \omega_\ell \left( x_1\cdots x_\ell \right) 
		\begin{vmatrix}
			1 &\ldots & 1 & \ldots & 1\\
			x_1^{n_1} &\ldots & x_i^{n_i} & \ldots &  x_\ell^{n_\ell}\\
			\vdots & & \vdots & & \vdots \\
			x_1^{(\ell-2)n_1} &\ldots &  x_i^{(\ell-2)n_i} & \ldots &  x_\ell^{(\ell-2)n_\ell}\\
			\frac{x_2^{n_2-1}\cdots x_\ell^{n_\ell-1}}{x_1}&\ldots & \frac{x_1^{n_1-1} \cdots \widehat{x_i^{n_i-1}} \cdots x_\ell^{n_\ell-1} }{x_i}& \ldots& 
			\frac{x_1^{n_1-1}\cdots x_{\ell-1}^{n_{\ell-1}-1}}{x_\ell}
		\end{vmatrix}\\
		=&p \omega_1 \cdots  \omega_\ell 
		\begin{vmatrix}
			1 &\ldots & 1 & \ldots & 1\\
			x_1^{n_1} &\ldots & x_i^{n_i} & \ldots &  x_\ell^{n_\ell}\\
			\vdots & & \vdots & & \vdots \\
			x_1^{(\ell-2)n_1} &\ldots &  x_i^{(\ell-2)n_i} & \ldots &  x_\ell^{(\ell-2)n_\ell}\\
			x_2^{n_2}\cdots x_\ell^{n_\ell}& \ldots & x_1^{n_1} \cdots \widehat{x_i^{n_i}} \cdots x_\ell^{n_\ell}& \ldots& 
			x_1^{n_1}\cdots x_{\ell-1}^{n_{\ell-1}}
		\end{vmatrix}\\
		=&p \omega_1 \cdots  \omega_\ell  \, (x_1^{n_1} \cdots x_\ell^{n_\ell}) \begin{vmatrix}
			1 &\ldots & 1 & \ldots & 1\\
			x_1^{n_1} &\ldots & x_i^{n_i} & \ldots &  x_\ell^{n_\ell}\\
			\vdots & & \vdots & & \vdots \\
			x_1^{(\ell-2)n_1} &\ldots &  x_i^{(\ell-2)n_i} & \ldots &  x_\ell^{(\ell-2)n_\ell}\\
			\frac{1}{x_1^{n_1}}& \ldots & \frac{1}{x_i^{n_i}}& \ldots& \frac{1}{x_\ell^{n_\ell}}
		\end{vmatrix}\\
		=&p \omega_1 \cdots  \omega_\ell \begin{vmatrix}
			x_1^{n_1} &\ldots & x_i^{n_i} & \ldots &  x_\ell^{n_\ell}\\
			x_1^{2n_1} &\ldots & x_i^{2n_i} & \ldots &  x_\ell^{2n_\ell}\\
			\vdots & & \vdots & & \vdots \\
			x_1^{(\ell-1)n_1} &\ldots &  x_i^{(\ell-1)n_i} & \ldots &  x_\ell^{(\ell-1)n_\ell}\\
			1 &\ldots & 1 & \ldots & 1\\
		\end{vmatrix}\\
		=&p \omega_1 \cdots  \omega_\ell (-1)^{\ell-1} \rom{det} \, V_{\ell-1},
	\end{align*}
	where $V_{\ell-1} := V_{\ell-1}[x_\ell^{n_\ell}, x_{\ell-1}^{n_{\ell-1}}, \ldots, x_1^{n_1} ]$  is a
	Vandermonde matrix in the $x_i^{n_i}$'s. Hence 
	\begin{align} \label{eqdetN_Fk}
		\rom{det} \, \begin{pmatrix}
			N_{T'_1, \ldots, T'_{\ell-2}}\\ \hdashline[1.5pt/2pt]
			T_{\ell-1}
		\end{pmatrix}
		&= \left( (-1)^{\ell-1} \omega_1 \cdots  \omega_\ell \right) p \cdot \prod_{1 \le i < j \le \ell}(x_j^{n_j} - x_i^{n_i})\\
		&= u F_{k} \nonumber
	\end{align}
	with $u$ a unit. By Saito's criterion, $\rom{Der}(F_{k})$ is a free $R$-module with basis, the $\ell$ derivations corresponding
	to the rows of the matrix  $\begin{pmatrix}
		N_{T'_1, \ldots, T'_{\ell-2}}\\ \hdashline[1.5pt/2pt]
		T_{\ell-1}
	\end{pmatrix}.$ Therefore by Definition \ref{defwFree}, $\rom{Der}(F_{k})$ is free with weighted exponents 
	$$(n_i\omega_i + 1, 2n_i\omega_i + 1, \ldots,  (\ell-2)n_i\omega_i + 1, (\ell-1)n_i\omega_i + 1 - \sum_{j = k+1}^\ell \omega_j)$$ for any 
	$i \in \{1, \ldots, \ell\}$.
\end{proof}

\begin{remark}
	This gives another proof of Theorem \ref{thmGeneralizedW} when $k = \ell$, $F_{\ell} = F$.
\end{remark}

The following corollaries are immediate by Lemma \ref{lemFreeWeightedEqualFreeAffine}.
 
\begin{corollary} \label{corGeneralizedAff}
	The divisor $\mathcal{D}_\ell^\ell(\{n_i\})$ in $\mathbb{A}^\ell$ defined by
	$$F = x_1 \cdots x_\ell \prod_{1 \le i < j \le \ell}(x_i^{n_i} - x_j^{n_j})$$
	is free for any positive integers $n_1, \ldots, n_\ell$. 
\end{corollary}

Recall that the divisor $\mathcal{D}_\ell^k(\{n_i\}) := \mathcal{D}_\ell^\ell(\{n_i\})/ x_{k+1}x_{k+2} \cdots x_\ell$ for $0 \le k \le \ell$.

\begin{corollary} \label{corAffremovelines}
	The divisor $\mathcal{D}_\ell^k(\{n_i\})$ in $\mathbb{A}^\ell$ defined by 
	$$F_k = p \cdot \prod_{1 \le i < j \le \ell}(x_i^{n_i} - x_j^{n_j}),\ \ \text{where} \ p = 
	\begin{cases}
		1, & \text{if}\  k = 0\\
		x_1 \cdots x_k, & \text{if}\ 1 \le k \le \ell
	\end{cases}$$
	is free for any positive integers $n_1, \ldots, n_\ell$.
\end{corollary}

Given the polynomials in Corollary \ref{corGeneralizedAff} and Corollary \ref{corAffremovelines}, we will now use the cone construction of 
Section \ref{secConeConstruct} to generate a family of divisors in
 standard projective $n$-space, which turns out to be free divisors and we give their exponents. But first, we need the following lemma,
 which will provide a new basis for $\rom{Der}(F)$ in Corollary \ref{corGeneralizedAff} and form part of a basis of $\rom{Der}(F_{k})$ in 
 Corollary \ref{corAffremovelines}.
 
 \begin{lemma} \label{lemAffDerv}
 	Let $F = x_1 \cdots x_\ell \prod_{1 \le i < j \le \ell}(x_i^{n_i} - x_j^{n_j})$ for any positive integers $n_q \in \mathbb{Z}_{>0}$ and let 
 	$\omega = (\omega_1, \ldots, \omega_\ell)$ be a weight vector such that $F$ is weighted homogeneous. Then for 
 	$1 \le m \le \ell$,
 	$\sum_{i = 1}^{m} \omega_i x_i \prod_{m < j \le \ell}(x_i^{n_i} - x_j^{n_j}) \partial_{x_i} \in \rom{Der}(F),$
 	and in particular, $\sum_{i = 1}^{m}(\prod_{j \ne i}n_j) x_i \prod_{m < j \le \ell}(x_i^{n_i} - x_j^{n_j}) \partial_{x_i} \in \rom{Der}(F).$ 
 	Moreover, they form a basis for $\rom{Der}(F)$.
 \end{lemma}
 
 \begin{proof}
 	Form the matrix corresponding to $\ell$
 	derivations in the statement of the lemma: when $m = \ell$, note that the corresponding derivation 
 	$\delta_{\mathcal{E}_\omega} = \sum_{i = 1}^\ell \omega_i x_i\partial_{x_i}$ is the weighted Euler derivation for the weight vector $\omega$. 
 	Starting with 
 	$m = \ell-1, \ell-2, \ldots, 1$, denote by $T_1, T_2, \ldots, T_{\ell-1}$, the tensors corresponding to the remaining $\ell-1$ derivations. 
 	By construction,
 	$T_i = ( g_1^i, g_2^i,\ldots, g^i_{\ell-i}, 0, \ldots, 0)$ for $1 \le i \le \ell-1$. Together, we form the matrix
 	$$M_{T_1, \ldots, T_{\ell-1}} = \begin{pmatrix}
 		\omega_1 x_1 & \omega_2 x_2 & \ldots &\ldots & \ldots& \omega_\ell x_\ell\\
 		g_1^1 & g_2^1 & \ldots & \ldots & g_{\ell-1}^1&0\\
 		g_1^2 & g_2^2 & \ldots & g_{\ell-2}^2&0&\vdots\\
 		\vdots & \vdots & \reflectbox{$\ddots$} & \reflectbox{$\ddots$}& \ldots & \vdots \\
 		g_1^{\ell-2} & g_2^{\ell-2} &0 & \ldots & \ldots & \vdots\\
 		g_1^{\ell-1} &0 &\ldots & \ldots & \ldots& 0
 	\end{pmatrix},$$
 	and similarly, form the matrix $N$ corresponding to the $\ell$ derivations in Lemma \ref{lem_vanDerMonde_der}:
 	$$N = \begin{pmatrix}
 		\omega_1 x_1 & \omega_2 x_2 & \ldots & \omega_\ell x_\ell\\
 		\omega_1 x_1^{n_1+1} & \omega_2 x_2^{n_2+1} & \ldots & \omega_\ell x_\ell^{n_\ell+1}\\
 		\vdots & \vdots & \ldots & \vdots \\
 		\omega_1 x_1^{(\ell-1)n_1+1} & \omega_2 x_2^{(\ell-1)n_2+1} & \ldots & \omega_\ell x_\ell^{(\ell-1)n_\ell+1}
 	\end{pmatrix}.$$
 	We claim that $N$ can be row reduced to $M_{T_1, \ldots, T_{\ell-1}}$ and vise versa, by elementary row operations over $\kk[x_1, \ldots, x_\ell]$.
 	
 	One considers matrix $N'$ obtained by factoring out $\omega_i x_i$ from each $i$-th column of $N$. Then
 	$\rom{det}(N) = (\prod_i \omega_i x_i) \rom{det}(N')$, and by permuting the columns of $N'$, swapping the first with the last column, second with the 
 	one before last, so on and so forth, one can write it as a Vandermonde matrix  
 	$V_{\ell-1}[x_\ell^{n_\ell}, x_{\ell-1}^{n_{\ell-1}}, \ldots, x_1^{n_1} ] = V_{\ell-1}$ in the $x_i^{n_i}$'s. 
 	By \cite[Theorem $2$]{luVandermondeYang}\footnote{\label{fnLUfield} The proof of \cite[Theorem $2$]{luVandermondeYang} is over the polynomial ring 
 		with field $\kk = \mathbb{R}$ but the proof actually works for any field $\kk$ or see \cite{Vandermonde_Rushanan} for instance.}, 
 		replacing $x_\ell^{n_\ell} \leftrightarrow x_0, \ldots, x_1^{n_1} \leftrightarrow x_{\ell-1}$ 
 	one can find a unique
 	LU factorization of $V_{\ell-1} = L_{\ell-1} U_{\ell-1}$ with coefficients in $\kk[x_1, \ldots, x_\ell]$,
 	where $L_{\ell-1}$ is a lower triangular matrix with $1$'s on its main diagonal, and $U_{\ell-1}$ is an upper-right triangular matrix. 
 	Permute the columns of $U_{\ell-1}$ so that it is an upper-left triangular matrix, denoted $U'$, by multiplying by $P$, the same permutation matrix used 
 	for $N'$ above to obtain a Vandermonde matrix: 
 	\begin{equation} \label{eqLUMtoN}
 		N' = V_{\ell-1}P = L_{\ell-1} U_{\ell-1}P = L_{\ell-1} U'.
 	\end{equation}
 	Note that by multiplying each $i$-th column of $U'$ by $\omega_i x_i$, the resulting matrix is $M_{T_1, \ldots, T_{\ell-1}},$ so we have 
 	$\rom{det}(M_{T_1, \ldots, T_{\ell-1}}) = (\prod_i \omega_i x_i) \rom{det}(U')$.
 	
 	The LU factorization shows that $U_{\ell-1} \sim V_{\ell-1}$ can be row reduced using only elementary row additions over $\kk[x_1, \ldots, x_\ell]$ 
 	because of the $1$'s on the main diagonal 
 	of $L_{\ell-1}$, hence $U' \sim N'$ also by elementary row additions by (\ref{eqLUMtoN}). Finally, multiplying the $i$-th columns of both $N'$ 
 	and $U'$ by $\omega_i x_i$ implies that 
 	$$M_{T_1, \ldots, T_{\ell-1}} \sim N$$ can be row reduced by elementary row operations. 
 	
 	Conversely, since in the LU factorization, $L_{\ell-1}$ is invertible over the polynomial ring 
 	$\kk[x_1, \ldots, x_\ell]$, it follows that one can reverse every operation, i.e. 
 	one writes  $L_{\ell-1}^{-1}V_{\ell-1} =  U_{\ell-1}$ 
 	and then applying the permutation matrix $P$: 
 	$$L_{\ell-1}^{-1} N' =  L_{\ell-1}^{-1}V_{\ell-1}P =  U_{\ell-1} P = U'.$$
 	Hence
 	$$N \sim M_{T_1, \ldots, T_{\ell-1}},$$ i.e. by elementary row operations over $\kk[x_1, \ldots, x_\ell]$, $N$ can be row reduced to $M_{T_1, \ldots, T_{\ell-1}}$.
 	
 	Since $N \sim M_{T_1, \ldots, T_{\ell-1}}$ using only elementary row operations, and since Lemma \ref{lem_vanDerMonde_der} says that the $\ell$
 	derivations corresponding to the rows of $N$ are in $\rom{Der}(F)$, this implies that \\
 	$\sum_{i = 1}^{m} \omega_i x_i\prod_{m < j \le \ell}(x_i^{n_i} - x_j^{n_j}) \partial_{x_i} \in \rom{Der}(F)$ for $1 \le m \le \ell$.
 	
 	It is clear also that $F$ is weighted homogenous when $\omega_i = \prod_{j \ne i} n_j$ for $1 \le i \le \ell.$
 	
 	Finally, the $\ell$ derivations in the lemma form a basis for $\rom{Der}(F)$ by Saito's criterion since by construction, the determinant of 
 	$M_{T_1, \ldots, T_{\ell-1}}$, its associated Saito matrix, is just the product along
 	the diagonal of its transpose, which is equal to $uF$ for $u = \prod_{i = 1}^\ell \omega_i \ge 1$ a unit.
 \end{proof}

\begin{theorem} \label{thmGeneralizedP}
	Let $n_1, \ldots, n_\ell$ be positive intergers such that $n_1 \le n_2 \le \ldots \le n_\ell$. The divisor $\widetilde{\mathcal{D}}_\ell^\ell(\{n_i\})$ in 
	projective space $\mathbb{P}^\ell$ defined by the homogeneous polynomial
	$$x_0 \cdot F^h = x_0 x_1 \cdots x_\ell \prod_{1 \le i < j \le \ell}(x_i^{n_i}x_0^{n_j - n_i} - x_j^{n_j})$$ 
	is free with exponents $(1, n_\ell + 1, n_\ell + n_{\ell-1} + 1, \ldots, n_\ell + n_{\ell-1} + \cdots +n_2 + 1)$, where $\_^h$  is homogenization in the variable $x_0$ in 
	$\kk[x_0, \ldots, x_\ell]$.
\end{theorem}

\begin{proof}
	Let $R = \kk[x_1, \ldots, x_\ell]$. The $\ell$ derivations in Lemma \ref{lemAffDerv} is a basis for $\rom{Der}(F) \subset \rom{Der}_{\kk}(R)$. Denoting the 
	basis by $\{\zeta_1, \ldots, \zeta_\ell\}$, ordered by increasing degree, 
	$$\sum_{i=1}^{\ell} \rom{deg} \, \zeta_i  = \ell + \sum^{\ell - 2}_{j=0} \sum^j_{i=0} n_{\ell-i} = 
	\rom{deg} \, F.$$ Hence by Proposition \ref{propConeConstructionFree} with $F$ reduced in $R$, the divisor in $\mathbb{P}^\ell$ 
	 from the cone construction defined by $x_0 \cdot F^h$ is free with exponents 
	 $$(\rom{deg} \, \zeta_1, \ldots, \rom{deg} \, \zeta_{\ell}) = (1, n_\ell + 1, n_\ell + n_{\ell-1} + 1, \ldots, n_\ell + n_{\ell-1} + \cdots +n_2 + 1).$$
\end{proof}

\begin{remark} \label{rmkBasesAffLiftRefArr}
	Lemma \ref{lem_vanDerMonde_der} and Lemma \ref{lemAffDerv} provide different bases for 
	$\rom{Der}(F)\subset \rom{Der}_{\kk}(R)$ where $R = \kk[x_1, \ldots, x_\ell]$; this is shown in the proofs of Theorem \ref{thmGeneralizedW} and 
	Lemma \ref{lemAffDerv}. Naturally the bases have derivations of equivalent weighted degrees in the context of weighted projective space $\mathbb{P}^{\ell-1}_\omega$ 
	(see Theorem \ref{thmGeneralizedW}). And in affine space $\mathbb{A}^\ell$, there are two bases generating $\rom{Der}(F)$ (see Corollary \ref{corGeneralizedAff}). 
	However, given $n_1 \le n_2 \le \ldots \le n_\ell$, in general if all $n_i$'s are not equal, then the
	basis given in Lemma \ref{thmGeneralizedW} cannot be lifted to $\mathbb{P}^\ell$ to form part of a basis of $\rom{Der}(x_0 F^h) \subset \rom{Der}_{\kk}(S)$ for 
	$S = \kk[x_0, \ldots, x_\ell]$, as opposed to the basis $\{\zeta_1, \ldots, \zeta_\ell\}$ of Lemma \ref{lemAffDerv} 
	(see Theorem \ref{thmGeneralizedP}) which always satisfies $\sum_{i=1}^{\ell} \rom{deg} \, \zeta_i  = \rom{deg} \,F.$
\end{remark}

We denote the divisor $\widetilde{\mathcal{D}}_\ell^k(\{n_i\}) := \widetilde{\mathcal{D}}_\ell^\ell(\{n_i\})/ x_{k+1}x_{k+2} \cdots x_\ell$ for $0 \le k \le \ell$.

\begin{theorem} \label{corProjremovelines}
	Let $n_1, \ldots, n_\ell$ be positive intergers such that $n_1 \le n_2 \le \ldots \le n_\ell$. 
	The divisor $\widetilde{\mathcal{D}}_\ell^k(\{n_i\})$ in $\mathbb{P}^\ell$ defined by the homogeneous polynomial 
	$$x_0 \cdot F_k^h = x_0 p \cdot \prod_{1 \le i < j \le \ell}(x_i^{n_i}x_0^{n_j - n_i} - x_j^{n_j}),
	\ \ \text{where} \ p = 
	\begin{cases}
		1, & \text{if}\  k = 0\\
		x_1 \cdots x_k, & \text{if}\ 1 \le k \le \ell
	\end{cases}$$ 
	is free with exponents 
	$$(1, n_\ell + 1, n_\ell + n_{\ell-1} + 1, ... ,n_\ell + n_{\ell-1} + \cdots +n_3 + 1, n_\ell + n_{\ell-1} + \cdots +n_2 + 1- \ell + k)$$
	 for all $0 \le k \le \ell$.
\end{theorem}

\begin{proof}
	Let $R = \kk[x_1, \ldots, x_\ell].$ Consider the matrix $M_{T_1, \ldots, T_{\ell-2}}$
	(a submatrix of $M_{T_1, \ldots, T_{\ell-1}}$ in the proof of Lemma \ref{lemAffDerv}) 
	corresponding to the $\ell-1$ $R$-linearly independent derivations in Lemma 
	\ref{lemAffDerv} for $2 \le m \le \ell$ and denote by $T_i= ( g_1^i, g_2^i,\ldots, g^i_{\ell-i}, 0, \ldots, 0)$ for $ 1 \le i \le \ell-2$, the tensors of the $2$nd to last row of 
	$M_{T_1, \ldots, T_{\ell-2}}$:
	\begin{equation}
		M_{T_1, \ldots, T_{\ell-2}} = \begin{pmatrix}
			\omega_1 x_1 & \omega_2 x_2 & \ldots &\ldots & \ldots& \omega_\ell x_\ell\\
			g_1^1 & g_2^1 & \ldots & \ldots & g_{\ell-1}^1&0\\
			g_1^2 & g_2^2 & \ldots & g_{\ell-2}^2&0&\vdots\\
			\vdots & \vdots & \reflectbox{$\ddots$} & \reflectbox{$\ddots$}& \ldots & \vdots \\
			g_1^{\ell-2} & g_2^{\ell-2} &0 & \ldots & \ldots & \vdots
		\end{pmatrix}.
	\end{equation} As in the proof of Theorem \ref{thmW_Dkl}, let $T_{\ell-1}= (g_1^{\ell-1}, g_2^{\ell-1},\ldots ,g_\ell^{\ell-1} )$ corresponding to the derivation 
	$\mu \in \rom{Der}(F_{k})$ in Lemma \ref{lemDerivMuOfFk}, where 
	$$g_i^{\ell-1} =  p \omega_i \cdot  x_1^{n_1-1}\cdots\widehat{x_i^{n_i-1}}\cdots x_\ell^{n_\ell-1} \ \text{for}\ 1 \le i \le \ell,$$
	and $T_{\ell-1}$ is added as a new row of the matrix $M_{T_1, \ldots, T_{\ell-2}}$. Once we show that the derivations corresponding to the rows of the matrix
	$M = \begin{pmatrix}
		M_{T_1, \ldots, T_{\ell-2}}\\ \hdashline[1.5pt/2pt]
		T_{\ell-1}\end{pmatrix}$ is basis for $\rom{Der}(F_k)  \subset \rom{Der}_{\kk}(R)$, then we will show that such a basis is a candidate to be lifted as in 
		Proposition \ref{propConeConstructionFree} to a basis of $\rom{Der}(x_0 \cdot F_k^h) \subset \rom{Der}_{\kk}(S)$ for $S = \kk[x_0, \ldots, x_\ell]$.
		
	First, recall the matrix $N_{T'_1, \ldots, T'_{\ell-2}}$ in \eqref{matrixN1_l-2} was built from the $\ell-1$ $R$-linearly independent derivations in Lemma 
	\ref{lem_vanDerMonde_der} for $0 \le m \le \ell-2$, and these derivations were shown to be part of a basis of $\rom{Der}(F_{k})$ in Theorem \ref{thmW_Dkl}. 
	Following the proof of Lemma \ref{lemAffDerv}, we will show that
	$$M_{T_1, \ldots, T_{\ell-2}} \sim N_{T'_1, \ldots, T'_{\ell-2}}$$ row reduced and vise versa, using only elementary row operations over $R$.
	For $1 \le i \le \ell$, factor $\omega_i x_i$ from the $i$th-columns of both $N_{T'_1, \ldots, T'_{\ell-2}}$ and $M_{T_1, \ldots, T_{\ell-2}}$, and denote the resulting 
	matrices $V$ and $U$ respectively.	
	In the LU factorization of the Vandermonde matrix $V_{\ell-1} = V_{\ell-1}[x_\ell^{n_\ell}, x_{\ell-1}^{n_{\ell-1}}, \ldots, x_1^{n_1} ]  = L_{\ell-1} U_{\ell-1}$ in 
	\cite[Theorem $2$]{luVandermondeYang}, and
	up to permutations of the columns by matrix $P$ (first column with last, second with next to last, etc.) as in the proof of Lemma \ref{lemAffDerv}, 
	\begin{equation}\label{eq_VPLUP} V_{\ell-1}P  = L_{\ell-1} (U_{\ell-1}P),\end{equation}
	and we note that $V$ is a submatrix of $V_{\ell-1}P$ where the last row is ommited, while $U$ is a submatrix of $U_{\ell-1}P$ with last row deleted.
	
	Now note that $L_{\ell-1}$ is a lower triangular matrix with $1$'s on its main diagonal.
	This implies (i) that all first $\ell-1$ rows of $V_{\ell-1}P$ are linear combinations of only the first $\ell-1$ rows of $U_{\ell-1}P$
	since the last column of $L_{\ell-1}$ has $0$'s except for the last entry, and (ii) that since there are $1$'s on the main diagonal of $L_{\ell-1}$, $V_{\ell-1}P$ can be 
	row reduced to $U_{\ell-1}P$ using only elementary row additions. This implies that 
	\begin{equation}\label{eqUsimV} U \sim V,\end{equation}
	i.e. $U$ can be row reduced to $V$ over 
	$\kk[x_1, \ldots, x_\ell]$ using only elementary row additions. Multiplying back the $i$-th columns of $U$ and $V$ by $\omega_i x_i$, this implies 
	$M_{T_1, \ldots, T_{\ell-2}} \sim N_{T'_1, \ldots, T'_{\ell-2}}$ row reduced and vise versa, using only elementary row operations since $L_{\ell-1}$ is invertible 
	over $\kk[x_1, \ldots, x_\ell]$ in equation (\ref{eq_VPLUP}), 
	as in the proof of Lemma \ref{lemAffDerv}.
	This implies that the $\ell-1$ derivations corresponding to the rows of $M_{T_1, \ldots, T_{\ell-2}}$ are also in $ \rom{Der}(F_{k})$.
	
	Next, we show that \begin{equation} \label{detME-schemeVanderMonde}
		\rom{det} \, \begin{pmatrix}
			M_{T_1, \ldots, T_{\ell-2}}\\ \hdashline[1.5pt/2pt]
			T_{\ell-1}
		\end{pmatrix} =
		\rom{det} \, \begin{pmatrix}
			N_{T'_1, \ldots, T'_{\ell-2}}\\ \hdashline[1.5pt/2pt]
			T_{\ell-1}
		\end{pmatrix},
	\end{equation}
	and until the end of the proof, we will consider matrices with coefficients in $\rom{Frac}(\kk[x_1, \ldots, x_\ell])$:
	\begin{align*}
		\rom{det} \, &\begin{pmatrix}
			N_{T'_1, \ldots, T'_{\ell-2}}\\ \hdashline[1.5pt/2pt]
			T_{\ell-1}
		\end{pmatrix} =\\
		&p \omega_1 \cdots  \omega_\ell \left( x_1\cdots x_\ell \right) 
		\begin{vmatrix}
			1 &\ldots & 1 & \ldots & 1\\
			x_1^{n_1} &\ldots & x_i^{n_i} & \ldots &  x_\ell^{n_\ell}\\
			\vdots & & \vdots & & \vdots \\
			x_1^{(\ell-2)n_1} &\ldots &  x_i^{(\ell-2)n_i} & \ldots &  x_\ell^{(\ell-2)n_\ell}\\
			\frac{x_2^{n_2-1}\cdots x_\ell^{n_\ell-1}}{x_1}&\ldots & \frac{x_1^{n_1-1} \cdots \widehat{x_i^{n_i-1}} \cdots x_\ell^{n_\ell-1} }{x_i}& \ldots& 
			\frac{x_1^{n_1-1}\cdots x_{\ell-1}^{n_{\ell-1}-1}}{x_\ell}
		\end{vmatrix}\\
		= &p \omega_1 \cdots  \omega_\ell \left( x_1\cdots x_\ell \right) \rom{det} \, \begin{pmatrix}
			V\\ \hdashline[1.5pt/2pt]
			T'_{\ell-1}
		\end{pmatrix},
	\end{align*}
	where $T'_{\ell-1} := (\frac{x_2^{n_2-1}\cdots x_\ell^{n_\ell-1}}{x_1},\ldots , \frac{x_1^{n_1-1} \cdots \widehat{x_i^{n_i-1}} \cdots x_\ell^{n_\ell-1} }{x_i}, \ldots, 
	\frac{x_1^{n_1-1}\cdots x_{\ell-1}^{n_{\ell-1}-1}}{x_\ell}).$ 
	Meanwhile,
	\begin{align*} 
		\rom{det} \, \begin{pmatrix}
			M_{T_1, \ldots, T_{\ell-2}}\\ \hdashline[1.5pt/2pt]
			T_{\ell-1}
		\end{pmatrix} = p \omega_1 \cdots  \omega_\ell \left( x_1\cdots x_\ell \right) \rom{det} \, \begin{pmatrix}
			U\\ \hdashline[1.5pt/2pt]
			T'_{\ell-1}
		\end{pmatrix}.
	\end{align*}
	To prove \eqref{detME-schemeVanderMonde}, simply note that by \eqref{eqUsimV}, $U$ can be row reduced to $V$ over $R$ using only elementary row additions.
	Hence, by \eqref{eqdetN_Fk} in the proof of Theorem \ref{thmW_Dkl}, 
	$$\rom{det} \, \begin{pmatrix}
		M_{T_1, \ldots, T_{\ell-2}}\\ \hdashline[1.5pt/2pt]
		T_{\ell-1}
	\end{pmatrix} = uF_k,$$ with $u$ a unit.
	
	Therefore $M = \begin{pmatrix}
		M_{T_1, \ldots, T_{\ell-2}}\\ \hdashline[1.5pt/2pt]
		T_{\ell-1}
	\end{pmatrix}$ is a Saito matrix for $F_k$. The first $\ell-1$ rows of $M$ recall correspond to the $\ell-1$ $R$-linearly 
	independent derivations in Lemma \ref{lemAffDerv} for $2 \le m \le \ell$, ordered by increasing degree. Denote them by $\zeta_1, \cdots, \zeta_{\ell-1}$ and denote 
	the last derivation corresponding to the tensor $T_{\ell-1}$ by $\mu_k$. Hence $\{ \zeta_1, \cdots, \zeta_{\ell-1}, \mu_k\}$ is a basis for 
	$\rom{Der}(F_k)\subset \rom{Der}_{\kk}(R)$ and 
	\begin{align*}
		\sum_{i=1}^{\ell-1} \rom{deg}\, \zeta_i + \rom{deg}\, \mu_k\ 
		&= ( \ell - 1 + \sum^{\ell - 3}_{j=0} \sum^j_{i=0} n_{\ell-i} ) + (k - \ell +1+ n_2+ \ldots + n_\ell)\\
		&= k + \sum^{\ell - 2}_{j=0} \sum^j_{i=0} n_{\ell-i} = \rom{deg}\, F_k,
	\end{align*}
	with $F_k$ reduced in $R$. By Proposition \ref{propConeConstructionFree}, 
	the divisor in $\mathbb{P}^\ell$ built from the cone construction and defined by $x_0 \cdot F_k^h$ is free with exponents 
	$$(\rom{deg} \, \zeta_1, \ldots, \rom{deg} \, \zeta_{\ell-1}, \rom{deg} \, \mu_k) = (1, n_\ell + 1, n_\ell + n_{\ell-1} + 1, \ldots, n_\ell + n_{\ell-1} + \cdots +n_2 + 1-\ell +k).$$
\end{proof}

\begin{remark}
	Making the connection with Remark \ref{rmkBasesAffLiftRefArr}, it was also shown in Theorem \ref{thmW_Dkl}'s proof that the derivations in Lemma \ref{lem_vanDerMonde_der} 
	 corresponding to the rows of $N_{T'_1, \ldots, T'_{\ell-2}}$ together with $\mu_k$ is also a basis for $\rom{Der}( F_k)$.
	  However, it cannot be used as a basis that can be lifted to form part of a basis of $\rom{Der}(x_0 \cdot F_k^h)$ (see Proposition \ref{propConeConstructionFree}'s proof) since their 
	  sum of degrees, as one can check, is in general not equal to $\rom{deg}(F_k)$.
\end{remark}

\begin{example}
	To illustrate with a family of free divisors in $\mathbb{P}^3$, giving explicit equations, take $\ell = 3$. By Theorem \ref{thmGeneralizedW}, for any positive integers $n_1, n_2, n_3$, 
	the divisor $\mathcal{D}_3^3(\{n_i\})$ in the weighted projective plane
	 $\mathbb{P}^2_\omega$, defined by the (weighted homogeneous) polynomial $$F = xyz(x^{n_1} - y^{n_2})(x^{n_1} - z^{n_3})(y^{n_2} - z^{n_3})$$ is free with weighted exponents 
	 $(n_i\omega_i + 1, 2n_i\omega_i + 1)$, where the weight vector $\omega = (\omega_1, \omega_2, \omega_3) = (n_2 n_3, n_1 n_3, n_1 n_2)$. 
	 
	 Hence, by Corollary \ref{corAffremovelines}, the 
	 corresponding divisor in affine $3$-space $\mathbb{A}^3$, also denoted $\mathcal{D}_3^3(\{n_i\})$ and defined by $F$ is free. Furthermore, denote by 
	 $\mathcal{D}_3^2(\{n_i\}), \mathcal{D}_3^1(\{n_i\})$ and $\mathcal{D}_3^0(\{n_i\})$, the divisors defined by the polynomials $xyf_{n_i}$, $xf_{n_i}$ and $f_{n_i}$ respectively, 
	 where $f_{n_i} := (x^{n_1} - y^{n_2})(x^{n_1} - z^{n_3})(y^{n_2} - z^{n_3}).$ By Corollary \ref{corAffremovelines}, $\mathcal{D}_3^2(\{n_i\}), \mathcal{D}_3^1(\{n_i\})$ and 
	 $\mathcal{D}_3^0(\{n_i\})$ are also free divisors in $\mathbb{A}^3$.
	 
	 By Theorem \ref{corProjremovelines}, letting the new variable $x_0 = t$ over $\kk[x,y,z,t]$, the corresponding divisors in standard projective $3$-space $\mathbb{P}^3$ obtained 
	 from the cone construction, are free with exponents given in the following table, for positive integers $n_1, n_2, n_3$ such that $n_1 \le n_2 \le n_3$:
	 \begin{center}
	 \begin{tabular}{ |p{1.3cm}|p{8.8cm}|p{3.5cm}|  }
	 	\hline
	 	\multicolumn{3}{|c|}{Free divisors in $\mathbb{P}^3$} \\
	 	\hline
	 	Divisor& Polynomial&Free exponents\\
	 	\hline
	 	$\widetilde{\mathcal{D}}_3^3(\{n_i\})$  & $txyz(x^{n_1}t^{n_2-n_1} - y^{n_2})(x^{n_1}t^{n_3-n_1} - z^{n_3})(y^{n_2}t^{n_3-n_2} - z^{n_3})$ 
	 	&$(1, n_3 + 1, n_3 + n_2+1)$\\
		$\widetilde{\mathcal{D}}_3^2(\{n_i\})$  & $txy(x^{n_1}t^{n_2-n_1} - y^{n_2})(x^{n_1}t^{n_3-n_1} - z^{n_3})(y^{n_2}t^{n_3-n_2} - z^{n_3})$ 
		&$(1, n_3 + 1, n_3 + n_2)$\\
		$\widetilde{\mathcal{D}}_3^1(\{n_i\})$  & $tx(x^{n_1}t^{n_2-n_1} - y^{n_2})(x^{n_1}t^{n_3-n_1} - z^{n_3})(y^{n_2}t^{n_3-n_2} - z^{n_3})$ 
		&$(1, n_3 + 1, n_3 + n_2-1)$\\
	 	$\widetilde{\mathcal{D}}_3^0(\{n_i\})$  & $t(x^{n_1}t^{n_2-n_1} - y^{n_2})(x^{n_1}t^{n_3-n_1} - z^{n_3})(y^{n_2}t^{n_3-n_2} - z^{n_3})$
	 	&$(1, n_3 + 1, n_3 + n_2-2)$\\
	 	\hline
	 \end{tabular}
	 \end{center}
\end{example}

\subsection{New free divisors in $\mathbb{P}^3$, a variant of Brieskorn-Pham polynomials example in \cite{Buchweitz_Conca_2012}}

In \cite{Buchweitz_Conca_2012}, Buchweitz and Conca studied triangular free divisors, which are divisors whose Saito matrix has a lower ``triangular'' shape and showed in particular that 
for $G_j = x_1^{r_1} + \cdots + x_j^{r_j}$ with $j = 2, \ldots, i$ and any positive integers $r_1, \ldots, r_i$,
the product $G_2 \cdots G_i$ of Brieskorn-Pham polynomials is a free divisor in \cite[Example $5.3$]{Buchweitz_Conca_2012}. We present a variant of their example and to our best 
knowledge new: $F_\Lambda$ is a product of Brieskorn-Pham polynomials that defines a family of free divisors that can be proved using the weighted Multiple eigenschemes techniques of 
Section \ref{secwME}.

The divisor defined by the reduced weighted homogeneous polynomial
$$F_\Lambda = (x^{r_0} + y^{r_1})\prod_{\alpha \in \Lambda} (x^{r_0} + y^{r_1} + \alpha z^{r_2}),$$ where $r_i$ are 
positive integers for $i = 0,1,2$ and $\Lambda$ a finite set of distinct non-zero elements $\alpha \in \kk$ in a field of char $0$ is a free divisor in
$\mathbb{P}^2_\omega$, as we shall prove shortly, with weighted exponents given below.

In particular, the homogeneous case gives a divisor in $\mathbb{P}^2$ defined by the reduced homogeneous polynomial 
$$F = (x^{r} + y^{r})\prod_{\alpha \in \Lambda} (x^{r} + y^{r} + \alpha z^{r})$$ that is free for all $r \in \mathbb{N}^*$ with exponents $(r-1, r|\Lambda|)$, where $|\Lambda|$ is the 
cardinality of the set. $F_\Lambda$ is a posteriori the corresponding weighted homogeneous polynomial of $F$  by changing the nonstandard $\mathbb{Z}$-grading on the coordinate ring 
$R = \kk[x,y,z]$.

\begin{theorem} \label{thmWBrieskorn-Pham}
	Let $\Lambda$ be a finite set of distinct non-zero elements $\alpha \in \kk$ in a field of char $0$. The divisor in $\mathbb{P}^2_\omega$ defined by the reduced weighted 
	homogeneous polynomial
	$$F_\Lambda = (x^{r_0} + y^{r_1})\prod_{\alpha \in \Lambda} (x^{r_0} + y^{r_1} + \alpha z^{r_2}),$$ 
	is free with weighted exponents 
	$(\omega_i r_i + 1 - \omega_0 - \omega_1, |\Lambda|\omega_i r_i  - \omega_2 + 1 ),$
	for any $i \in \{0,1,2\}$ and $r_0, r_1, r_2$ are any positive integers.
\end{theorem}

\begin{proof}
	
Observe that $F_\Lambda$ is reduced and weighted homogeneous w.r.t. the weight $\omega = (\omega_0, \omega_1, \omega_2) = (r_1r_2, r_0r_2, r_0r_1)$ for instance, and one considers the 
weighted homogeneous derivation $$\delta_1 = - r_1 y^{r_1 - 1} \partial_x + r_0 x^{r_0 - 1}\partial_y.$$ The later satisfies 
$\delta_1(x^{r_0} + y^{r_1} + \alpha z^{r_2}) = \delta_1(x^{r_0} + y^{r_1}) = 0$ for all $\alpha \in \Lambda$, which implies $\delta_1 \in \rom{Der}_0(F_\Lambda)$. The 
weighted eigenscheme $E(T_1)$ of the associated tensor $T_1=(- r_1 y^{r_1 - 1}, r_0 x^{r_0 - 1}, 0)$ of $\delta_1$ is the closed subscheme of $\mathbb{P}^2_\omega$, defined by 
the ideal $I_{E(T_1)}$ of maximal minors of the matrix
$$M_{T_1} = \begin{pmatrix}
	\omega_0 x & \omega_1 y & \omega_2 z\\
	- r_1 y^{r_1 - 1} & r_0 x^{r_0 - 1} & 0
\end{pmatrix}.$$
Hence $I_{E(T_1)} = (x^{r_0} + y^{r_1}, zy^{r_1}, zx^{r_0})$ has codimension $2$ by Lemma \ref{lemGcdCodim}. Using Theorem \ref{propMEscheme}, since
$F_\Lambda \in I_{E(T_1)}$, this implies that the divisor $V(F_\Lambda)$ is free with weighted exponents
$$(d_1, \rom{deg}_\omega(F_\Lambda )+ \sum_{i=0}^2 (1-\omega_i) - d_1 -1),$$ where 
$d_1 := \rom{deg}_\omega(\delta_1) = \omega_i r_i + 1 - \omega_0 - \omega_1$, and $\rom{deg}_\omega(F_\Lambda ) = (|\Lambda|+1)\omega_i r_i$
for any $i \in \{0,1,2\},$ so $\rom{deg}_\omega(F_\Lambda )+ \sum_{i=0}^2 (1-\omega_i) - d_1 -1 = |\Lambda|\omega_i r_i  - \omega_2 + 1.$
\end{proof}

	By Lemma \ref{lemFreeWeightedEqualFreeAffine}, it follows that $F_\Lambda$ also defines a free divisor in $\mathbb{A}^3$.
	
	\begin{corollary} \label{corBrieskorn-Pham}
		Let $\Lambda$ be any finite set of non-zero distinct elements over the base field $\kk$. The divisor in $\mathbb{A}^3$ defined by the polynomial
		$$F_\Lambda = (x^{r_0} + y^{r_1})\prod_{\alpha \in \Lambda} (x^{r_0} + y^{r_1} + \alpha z^{r_2})$$
		in $\kk[x,y,z]$ is free for all positive integers $r_0, r_1, r_2.$
	\end{corollary}
	
	By the cone construction, one can consider the corresponding divisor in $\mathbb{P}^3$ defined by the reduced homogeneous polynomial $t \cdot F_\Lambda^h \in \kk[x,y,z,t]$. As earlier, denote
	by $\_^h$ homogenization in the variable $t$. We have the following theorem.
	
	\begin{theorem} \label{thmBrieskorn-PhamProj}
		Let $\Lambda$ and $F_\Lambda$ be as defined above and $r_0, r_1, r_2$ any positive integers. 
		Then the divisor in $\mathbb{P}^3$ defined by the homogeneous polynomial
		$t \cdot F_\Lambda^h \in \kk[x,y,z,t]$
		is free with exponents $(1, \rom{max}\{r_0,r_1\} - 1, |\Lambda|\cdot \rom{max}\{r_0,r_1, r_2\} ).$
		
		In particular, if $r_0 \le r_1 \le r_2$, then the divisor defined in $\mathbb{P}^3$ by
		$$ t (t^{r_1-r_0}x^{r_0} + y^{r_1})\prod_{\alpha \in \Lambda} (t^{r_2-r_0}x^{r_0} + t^{r_2-r_1}y^{r_1} + \alpha z^{r_2})$$
		is free with exponents $(1,r_1 - 1, |\Lambda|r_2)$.
	\end{theorem}
	
	\begin{proof}
	Suppose $r_0 \le r_1 \le r_2$ (the proofs of the other orderings, being similar to the present one, are left to the interested reader). 
	Let $\delta_1 = - r_1 y^{r_1 - 1} \partial_x + r_0 x^{r_0 - 1}\partial_y$ and the weighted Euler derivation $\delta_{\mathcal{E}_\omega}$ w.r.t. $\omega = (r_1r_2, r_0r_2, r_0r_1)$  
	be as in the proof of Theorem \ref{thmWBrieskorn-Pham}.
	By Corollary \ref{corHomogenizingDerv}, $\delta_{\mathcal{E}_\omega}$
	and $\delta_1^h$ are in $\rom{Der}(F_\Lambda^h)$, which implies $\delta_{\mathcal{E}_\omega},\delta_1^h \in \rom{Der}(t\cdot F_\Lambda^h)$ since the derivations have no term in $\partial_t$. 
	The ME-scheme matrix is
	$$M_{T_1, T_2} = \begin{pmatrix}
		x & y & z & t\\
		r_1 r_2 x & r_0 r_2 y & r_0 r_1 z & 0\\
		-r_1 y^{r_1-1} & r_0 t^{r_1 - r_0} x^{r_0 -1} &0&0
	\end{pmatrix},$$
	and 
	\begin{align*}
		I_{E(T_1, T_2)} = \langle &t(t^{r_1 -r_0}x^{r_0} + y^{r_1}),\\
		 &t^{r_1-r_0+1}zx^{r_0 - 1},\\
		 &tzy^{r_1 -1},\\ 
		 &z(r_0r_1r_2 t^{r_1 -r_0}x^{r_0} + r_0r_1r_2 y^{r_1}) - r_0 r_1 z(r_0 t^{r_1-r_0}x^{r_0} + r_1 y^{r_1})\rangle\\
		 =: &(m_1,m_2,m_3,m_4).
	\end{align*}
	Note that $\rom{gcd}(m_1,m_2,m_3,m_4) = 1,$ which implies by Lemma \ref{lemGcdCodim} that $I_{E(T_1, T_2)}$ has $\text{codimension} \, 2$.
	Since $t \cdot F_\Lambda^h = t(t^{r_1 -r_0}x^{r_0} + y^{r_1}) \prod_{\alpha \in \Lambda} (t^{r_2 - r_0}x^{r_0} + t^{r_2 - r_1}y^{r_1} + \alpha z^{r_2})$, clearly 
	$t \cdot F_\Lambda^h \in I_{E(T_1, T_2)},$ and by \cite[Proposition $3.3$]{digennaro2025saitostheoremrevisitedapplication}, the divisor defined by $t \cdot F_\Lambda^h$ is free with exponents 
	$(1,r_1 - 1, |\Lambda|r_2)$. 
	\end{proof}

\subsection{New free divisors coming from pencils of hypersurfaces example in \cite{digennaro2025saitostheoremrevisitedapplication}}

	Let $R = \kk[x_0, \ldots, x_n]$ denote the coordinate ring for $\mathbb{P}^n_\omega$ and $\mathbb{A}^{n+1}$ and $S = \kk[x_0, \ldots, x_n, x_{n+1}]$ for $\mathbb{P}^{n+1}$,
	where $\kk$ is a field of char $0$.

	In \cite[Theorem $4.2$]{digennaro2025saitostheoremrevisitedapplication}, the following family of free divisors in $\mathbb{P}^n$ is given:
	starting with the free hyperplane arrangement $\mathcal{A} : x_0 \cdots x_n = 0$, one considers its Jacobian ideal
	$$J = (h_0, h_1, \ldots, h_n),$$
	where $h_i := x_0 \cdots \hat{x_i}\cdots x_n$, for $0 \le i \le n$, with singular locus the ${{n+1}\choose{2}} $ codimension two
	faces of the hypertetrahedron. Fix $m_0$ such that $n = 2m_0 + \epsilon$ where $\epsilon = 0$ or $1$. Then the defining polynomials of two hypersurfaces
	$S_1$ and $S_2$ are given respectively as
	$$f_1 = \sum_{i = m_0+1}^n h_i \text{ and } f_2 = \sum_{i = 0}^{m_0} h_i.$$
	Furthermore, consider the pencil of hypersurfaces $C(f_1,f_2) = \{ S_{a,b} = a S_1 + b S_2\}_{[a;b] \in \mathbb{P}^1}$ of degree $n$.

	 In \cite[Theorem $4.2$]{digennaro2025saitostheoremrevisitedapplication}, Di Gennaro and Miró-Roig showed that the divisors in $\mathbb{P}^n$ 
	 \begin{enumerate}
	 	\item $S_1S_2$ defined by $f_1f_2$ is free with exponents $(1, 2,\cdots, 2)$, and
	 	\item $S_1S_2 \prod_{i = 3}^k (a_i S_1 + b_i S_2)$, defined by $f_1f_2\prod_{i = 3}^k (a_i f_1 + b_i f_2)$, where $[a_i;b_i] \in  \mathbb{P}^1$, $k \ge 3$
	 	and $a_i S_1 + b_i S_2$ are generic members of the pencil $C(f_1,f_2)$, is free with exponents $(2,\cdots, 2, n(k- 2) + 1)$.\\
	 \end{enumerate}
	 
	 One way to generalize example \cite[Theorem $4.2$]{digennaro2025saitostheoremrevisitedapplication} is to let $m_0$ vary between $0$ and
	  $n-1$, and to take the $r^{th}$ power of each variable where $r$ is any positive integer, i.e. let 
	 $$h'_i := x_0^r \cdots \widehat{x_i^r}\cdots x_n^r,$$ 
	 for $0 \le i \le n$ and $r \in \mathbb{Z}_{>0}$, and for \textit{any} choice of $m$ where $0 \le m \le n-1$, we define 
	 $$f'_1 = \sum_{i = m+1}^n h'_i \text{ and } f'_2 = \sum_{i = 0}^{m} h'_i,$$
	the two homogeneous polynomials of degree $rn$ of the hypersurfaces $S'_1$ and $S'_2$ respectively.\\
	
	\begin{notation}
		For a polynomial $f$, denote $f^{\rom{red}}$ its reduced polynomial, i.e. product of square-free irreducible factors.
		Similarly, for a divisor $D$, denote $D^{\rom{red}}$ its corresponding reduced divisor. 
	\end{notation} 
	
	\begin{corollary} \label{corf1f2powerN}
		With the above notation, the divisors in $\mathbb{P}^n$
		\begin{enumerate}
			\item $(S'_1S'_2)^{\rom{red}}$ defined by 
			$$(f'_{1}f'_{2})^{\rom{red}} = x_0 \cdots x_n 
			\left( \sum_{i = 0}^m x_{0}^{r} \cdots \widehat{x_i^{r}}\cdots x_m^{r} \right)
			\left( \sum_{i = m+1}^n x_{m+1}^{r} \cdots \widehat{x_i^{r}}\cdots x_n^{r} \right)$$
			 is free with exponents $(1, r+1,\cdots, r+1)$, and
			\item $(S'_1S'_2)^{\rom{red}} \prod_{i = 3}^k (a_i S'_1 + b_i S'_2)$, defined by $$(f'_1f'_2)^{\rom{red}}\prod_{i = 3}^k (a_i f'_1 + b_i f'_2),$$ for $k \ge 3$, $[a_i;b_i] \in  \mathbb{P}^1$,
			and $a_i S'_1 + b_i S'_2$ being generic members of the pencil $C(f'_1,f'_2)$, is free with exponents $(r+1,\cdots, r+1, rn(k - 2) + 1)$.\\
		\end{enumerate}
	\end{corollary}
	
	Corollary \ref{corf1f2powerN} follows from a more general theorem in weighted projective space $\mathbb{P}^n_\omega$, Theorem \ref{thmf1f2weighted}, which we prove below.
	As was the case previously, by changing the nonstandard $\mathbb{Z}$-grading on the ring $R$ allows one to consider the 
	corresponding weighted homogeneous polynomials $\widetilde{f_1}$ and $\widetilde{f_2}$:
	 
	 Let
	 $$\widetilde{h_i} := x_0^{r_0} \cdots \widehat{x_i^{r_i}}\cdots x_n^{r_n},$$
	 for $0 \le i \le n$ and $r_0, \ldots, r_n$ are any positive integers, and for \textit{any} choice of $m$ where $0 \le m \le n-1$, define
	 \begin{equation} \label{eqf1f2tilde}
	 	\widetilde{f_1} = \sum_{i = m+1}^n \widetilde{h_i} \text{ and } \widetilde{f_2} = \sum_{i = 0}^{m} \widetilde{h_i}.
	 \end{equation}
	 
	\begin{remark}
		For any $m \in \{0,1,\ldots, n-1\}$, and any $r_i \in \mathbb{Z}_{>0}$ for $0 \le i \le n$, 
		$\widetilde{f_1}$ and $\widetilde{f_2}$
		are weighted homogeneous polynomials for some weight vector. Moreover, there exist a weight vector $\omega$ such that
		$\widetilde{f_1}$ and $\widetilde{f_2}$ are of same weighted degree $\omega$.
		Indeed note that weight vector $\omega = (\omega_0, \ldots, \omega_n)$ with $\omega_i = r_0 \ldots \widehat{r_i} \ldots r_n$ for $0 \le i \le n$ works.
		
		Note that since $\widetilde{f_1} = \sum_{i = m+1}^n \widetilde{h_i}$ is the sum of $n-m$ monomials, each
		differing at one variable in $\{x_{m+1}, \ldots, x_n\}$, $\widetilde{f_1}$ is weighted homogeneous if and only if 
		\begin{equation*}r_i \omega_i = r_j \omega_j, \text{ for } m+1 \le i < j \le n, \end{equation*}
		and similarly for $\widetilde{f_2}$, we have
		\begin{equation*}r_i \omega_i = r_j \omega_j, \text{ for } 0 \le i < j \le m.\end{equation*}
		
		If in addition, $\widetilde{f_1}$ and $\widetilde{f_2}$ are weighted homogeneous of the same degree, then
		\begin{equation}r_i \omega_i = r_j \omega_j, \text{ for all } i,j \in \{0, \ldots, n\}. \label{eqromegaequal}\end{equation}
	\end{remark}
	\begin{remark} \label{rmkf1f2reduced}
		Moreover, note that by construction of $\widetilde{f_1}$ and $\widetilde{f_2}$, it follows that
		$ (\widetilde{f_1} \widetilde{f_2})^{\rom{red}} = \widetilde{f_1} \widetilde{f_2}/(x_0^{r_0-1} \cdots x_n^{r_n-1}).$
	\end{remark}

	\begin{theorem} \label{thmf1f2weighted} Choose $m \in \{0,1,\ldots, n-1\}$ and let $r_0, \ldots, r_n$ be any positive integers; this defines
		$\widetilde{f_1}$ and $ \widetilde{f_2}$ as constructed in (\ref{eqf1f2tilde}). 
		Let $\omega = (\omega_0, \ldots, \omega_n)$ be a weight vector such that
		 $\widetilde{f_1}$ and $\widetilde{f_2}$ are weighted homogeneous of the same degree $r_i\omega_i n$ for all $i$.
		\begin{enumerate}
			\item The divisor $(\widetilde{S_1}\widetilde{S_2})^{\rom{red}}$ in weighted projective space $\mathbb{P}^n_\omega$ defined by the weighted
			homogeneous polynomial
			$$(\widetilde{f_1}\widetilde{f_2})^{\rom{red}} = x_0 \cdots x_n 
			\left( \sum_{i = 0}^m x_{0}^{r_{0}} \cdots \widehat{x_i^{r_i}}\cdots x_m^{r_m} \right)
			\left( \sum_{i = m+1}^n x_{m+1}^{r_{m+1}} \cdots \widehat{x_i^{r_i}}\cdots x_n^{r_n} \right)$$
			is free with weighted exponents $(1, r_i\omega_i + 1, \ldots, r_i\omega_i+ 1)$.
			\item The divisor $(\widetilde{S_1}\widetilde{S_2})^{\rom{red}} \prod_{i = 3}^k (a_i \widetilde{S_1} + b_i \widetilde{S_2})$ in weighted projective space $\mathbb{P}^n_\omega$ 
			defined by the weighted homogeneous polynomial
			 $$\left( \widetilde{f_1}\widetilde{f_2}\right)^{\rom{red}}\prod_{i = 3}^k (a_i \widetilde{f_1} + b_i \widetilde{f_2}),$$
			  for $k \ge 3$, $[a_i;b_i] \in  \mathbb{P}^1$,
			and $a_i \widetilde{S_1} + b_i \widetilde{S_2}$ being generic members of the pencil $C(\widetilde{f_1},\widetilde{f_2})$, is free with
			 weighted exponents $$(r_i\omega_i +1,\cdots, r_i\omega_i +1, (k-2)nr_i\omega_i + 1).$$\\
			
		\end{enumerate}
	\end{theorem}
	
	\begin{proof}
		Fix $m \in \{0,1,\ldots, n-1\}$ and let $\delta_{\mathcal{E}_\omega} = \sum^n_{i = 0} \omega_i x_i \partial_{x_i}$ be the weighted Euler derivation.\\
		(i) From
		\begin{equation*}
			\partial_{x_k}(\widetilde{f_1})=
			\begin{cases}
				r_k \widetilde{f_1}/x_k, & \text{if}\ 0 \le k \le m \\
				r_k(\widetilde{f_1} - \frac{x_0^{r_0} \cdots x_n^{r_n}}{x_k^{r_k}})/x_k, & \text{else}\  m+1 \le k \le n
			\end{cases}
		\end{equation*}
		\begin{equation*}
			\partial_{x_k}(\widetilde{f_2})=
			\begin{cases}
				r_k(\widetilde{f_2} - \frac{x_0^{r_0} \cdots x_n^{r_n}}{x_k^{r_k}})/x_k, & \text{if}\ 0 \le k \le m \\
				r_k \widetilde{f_2}/x_k, & \text{else}\  m+1 \le k \le n
			\end{cases}
		\end{equation*}
		one constructs derivations of $\rom{Der}(\widetilde{f_1}) \cap \rom{Der}(\widetilde{f_2}) = \rom{Der}((\widetilde{f_1}\widetilde{f_2})^{\rom{red}})$ by Remark \ref{rmkDerIntersection}:
		\begin{align}
		&\frac{1}{r_i} x_i^{r_i + 1} \partial_{x_i} - \frac{1}{r_j} x_j^{r_j + 1} \partial_{x_j} \ \text{for} \ i,j \in \{0, \ldots, m\} \ \text{with} \ i\ne j, \label{eqderivationsB1block}\\
		&\frac{1}{r_s} x_s^{r_s + 1} \partial_{x_s} - \frac{1}{r_t} x_t^{r_t + 1} \partial_{x_t} \ \text{for} \ s,t \in \{m+1, \ldots, n\} \ \text{with} \ s\ne t. \label{eqderivationsB2block}
		\end{align}
		For $\widetilde{f_a}$ with $a = 1,2$, direct calculation shows that 
		\begin{align}
			\left(\frac{1}{r_i} x_i^{r_i + 1} \partial_{x_i} - \frac{1}{r_j} x_j^{r_j + 1} \partial_{x_j}\right)(\widetilde{f_a}) 
			\label{derivfa1}
			&= (x_i^{r_i} - x_j^{r_j})\widetilde{f_a} \in \rom{Der}(\widetilde{f_a}),\\
			&\nonumber  \text{for} \  i,j \in \{0, \ldots, m\},\ \text{with} \ i\ne j, \ \text{and}\\
			\left(\frac{1}{r_s} x_s^{r_s + 1} \partial_{x_s} - \frac{1}{r_t} x_t^{r_t + 1} \partial_{x_t} \right)(\widetilde{f_a}) 
			\label{derivfa2}
			&= (x_s^{r_s} - x_t^{r_t})\widetilde{f_a} \in \rom{Der}(\widetilde{f_a}),\\
			&\nonumber \text{for} \ s,t \in \{m+1, \ldots, n\},\ \text{with}  \ s\ne t,
		\end{align} 
		each of same weighted degree $r_i\omega_i + 1$ for any $i\in \{0, \ldots, n\}$ by (\ref{eqromegaequal}) and the fact that 
		$\rom{deg}_\omega(\partial_{x_i}) = 1-\omega_i$.

		In (\ref{eqderivationsB1block}), there are ${{m+1}\choose{2}}$ derivations, so to obtain $m$ derivations, whose corresponding  
		tensors $R_j = (R_j^0, R_j^1, \ldots, R_j^n)$ are $\kk$-linearly independent, we fix any $i \in  \{0,1,\ldots, m\}$ and then take the remaining $m$ choices 
		for $j \in  \{0,1,\ldots, m\}$ with $j \ne i$. Choose $i=0$ and let
		\begin{equation}\label{derivationsB1}
		\delta_j = \frac{1}{r_0} x_0^{r_0 + 1} \partial_{x_0} - \frac{1}{r_j} x_j^{r_j + 1} \partial_{x_j} \ \text{for} \ 1 \le  j \le m.
		\end{equation}
		Similarly, among ${{n-m}\choose{2}}$ derivations in (\ref{eqderivationsB2block}), we fix $s \in \{m+1,\ldots, n\}$ and obtain $n-m-1$ derivations with
		tensors $S_t = (S_t^0, S_t^1, \ldots, S_t^n)$.
		Choose $s = m+1$ and let
		\begin{equation}\label{derivationsB2}
			\eta_t = \frac{1}{r_{m+1}} x_{m+1}^{r_{m+1} + 1} \partial_{x_{m+1}} - \frac{1}{r_t} x_t^{r_t + 1} \partial_{x_t} \ \text{for} \ m+2 \le t \le n.
		\end{equation}
		
		With the corresponding $n-1$ tensors $$R_j = \left(\frac{1}{r_0} x_0^{r_0 + 1}, 0, \ldots, 0, - \frac{1}{r_j} x_j^{r_j + 1}, 0 \ldots, 0\right) \ \text{for} \ 1 \le  j \le m,$$ 
		and $$S_t = \left(0, \ldots, 0, \frac{1}{r_{m+1}} x_{m+1}^{r_{m+1} + 1}, 0, \ldots, 0, - \frac{1}{r_t} x_t^{r_t + 1}, 0, \ldots, 0\right)
		\ \text{for} \ m+2 \le t \le n,$$ 
		one may build the matrix $M_{R_1, \ldots, S_n }$ of ME-scheme associated to the $n-1$ derivations (\ref{derivationsB1}) and (\ref{derivationsB2}).
		
		Let \begin{equation} \label{derivationLastmu} \mu = \sum_{i = m+1}^n \frac{1}{r_i} x_i \partial_{x_i}, \end{equation} 
		and note that $\mu \in \rom{Der}(\widetilde{f_1}) \cap \rom{Der}(\widetilde{f_2})= \rom{Der}((\widetilde{f_1}\widetilde{f_2})^{\rom{red}}):$
		\begin{align*}
			\mu(\widetilde{f_1}) &= \sum_{i = m+1}^n \left( \widetilde{f_1} - \frac{x_0^{r_0} \cdots x_n^{r_n}}{x_{i}^{r_{i}}} \right) 
			&\mu(\widetilde{f_2}) = (n-m)\widetilde{f_2}.\\
			&= (n-m)\widetilde{f_1} - \sum_{i = m+1}^n \frac{x_0^{r_0} \cdots x_n^{r_n}}{x_{i}^{r_{i}}}&\\
			&= (n-m)\widetilde{f_1} - \widetilde{f_1} =  (n-m-1)\widetilde{f_1},&\\
		\end{align*}
		This is the last derivation of weighted degree one, whose tensor we add to the matrix $M_{R_1, \ldots, S_n }$, to form the 
		Saito matrix $M$.
		Next we show that
		$$\rom{det} \ M = \frac{r_i \omega_i}{r_0 \cdots r_n} \rom{det} 
		 \begin{pmatrix}\begin{array}{c;{2pt/2pt}c}
		 		B_1 & 0\\ \hdashline[2pt/2pt]
		 		0 & B_2 
		 \end{array}\end{pmatrix} = u (\widetilde{f_1} \widetilde{f_2})^{\rom{red}},$$
		 where $u = (-1)^{n-1} \frac{r_i \omega_i}{r_0 \cdots r_n} \ne 0$ is a unit. First, factor $\frac{1}{r_i}$ from each column, then
		 by (\ref{eqromegaequal}), one can factor out $r_i \omega_i$ from the first row corresponding to the weighted Euler derivation.
		 Next we compute $\rom{det}\begin{pmatrix}\begin{array}{c;{2pt/2pt}c}
		 	B_1 & 0\\ \hdashline[2pt/2pt]
		 	0 & B_2 
		 \end{array}\end{pmatrix}$, where
		 \begin{align*}
		 	B_1 &= \begin{pmatrix}
		 		&x_0 & x_1 & \cdots&\cdots&x_{m-1}& x_m \\
		 		 &x_0^{r_0 + 1} &-x_1^{r_1 + 1} & 0 &\cdots&\cdots&0\\
		 		&x_0^{r_0 + 1} &0&-x_2^{r_2 + 1} &0&\cdots&0\\
		 		&\vdots&&&&\ddots&\vdots\\
		 		&x_0^{r_0 + 1}&0&\cdots&\cdots&0& -x_m^{r_m + 1}
		 	\end{pmatrix} \text{ and}\\
			B_2 &= \begin{pmatrix}
				&x_{m+1}^{r_{m+1} + 1}&-x_{m+2}^{r_{m+2} + 1}&0&\cdots&\cdots&0\\
				&x_{m+1}^{r_{m+1} + 1}&0&-x_{m+3}^{r_{m+3} + 1}&0&\cdots&0\\
				&\vdots&&&&\ddots&\vdots\\
				 &x_{m+1}^{r_{m+1} + 1}&0&\cdots&\cdots&0&-x_n^{r_n + 1}\\
				 &x_{m+1}& x_{m+2}&\cdots&\cdots&x_{n-1}& x_n
			\end{pmatrix}.
		\end{align*}
		Let $B_1 = (b_{i,j})$ with $1 \le i,j \le m+1$, then expanding along the first column, one finds 
		\begin{align*}
		\rom{det}\, B_1 
		&=  x_0 \cdots x_m \left( \sum_{i = 0}^m (-1)^m x_0^{r_0} \cdots \widehat{x_i^{r_i}} \cdots x_m^{r_m} \right)
		= (-1)^m \frac{x_0 \cdots x_m}{x_{m+1}^{r_{m+1}} \cdots x_n^{r_n}}\widetilde{f_2}.
		\end{align*}
	 Similarly,
		\begin{align*}\rom{det}\, B_2 &= x_{m+1} \cdots x_n \left( \sum_{i = m+1}^n (-1)^{n-m-1} x_{m+1}^{r_{m+1}} \cdots \widehat{x_i^{r_i}} \cdots x_n^{r_n} \right)
		=(-1)^{n-m-1} \frac{x_{m+1} \cdots x_n}{x_{0}^{r_{0}} \cdots x_m^{r_m}}\widetilde{f_1}.
		\end{align*} 
		Therefore 
		\begin{align*}\rom{det}\begin{pmatrix}\begin{array}{c;{2pt/2pt}c}
		B_1 & 0\\ \hdashline[2pt/2pt]
		0 & B_2 
	\end{array}\end{pmatrix} &= (-1)^{n-1} \frac{x_0 \cdots x_n}{x_{0}^{r_{0}} \cdots x_n^{r_n}} \widetilde{f_1}\widetilde{f_2}
		= (-1)^{n-1} (\widetilde{f_1} \widetilde{f_2})^{\rom{red}}
		\end{align*}
		 by Remark \ref{rmkf1f2reduced}. Finally, by Saito's criterion, 
		  $\delta_{\mathcal{E}_\omega}, \delta_j \ \text{for} \ 1 \le  j \le m$ in (\ref{derivationsB1}), $\eta_t\ \text{for} \ m+2 \le t \le n$ in 
		 (\ref{derivationsB2}) and $\mu$ form a basis for $\rom{Der}((\widetilde{f_1} \widetilde{f_2})^{\rom{red}}).$ Moreover, this also shows that
		 $E(R_1, \ldots, S_n )$ has codimension $2$ and that $(\widetilde{f_1} \widetilde{f_2})^{\rom{red}} \in I_{E(R_1, \ldots, S_n )},$ which by 
		 Theorem \ref{propMEscheme} shows that $\rom{Der}((\widetilde{f_1} \widetilde{f_2})^{\rom{red}})$ is free with weighted exponents $(1, r_i\omega_i + 1, \ldots, r_i\omega_i+ 1).$ 
		 This proves (i).\\
		 
		 (ii) Clearly, $\delta_{\mathcal{E}_\omega} \in \rom{Der}((\widetilde{f_1} \widetilde{f_2}\prod_{i = 3}^k (a_i \widetilde{f_1} + b_i \widetilde{f_2}))^{\rom{red}})$ for  
		 $k \ge 3$. 
		 Let $a_i \widetilde{S_1} + b_i \widetilde{S_2}= V(a_i \widetilde{f_1} + b_i \widetilde{f_2})$ be a generic member of the pencil. Note that (\ref{derivationsB1}) derivations
		 $$\delta_j = \frac{1}{r_0} x_0^{r_0 + 1} \partial_{x_0} - \frac{1}{r_j} x_j^{r_j + 1} \partial_{x_j} \in \rom{Der}(a_i \widetilde{f_1} + b_i \widetilde{f_2}) 
		 \ \text{for} \ 1 \le  j \le m,$$
		 since by (\ref{derivfa1}), $\delta_j (a_i \widetilde{f_1} + b_i \widetilde{f_2})= a_i (x_i^{r_i} - x_j^{r_j}) \widetilde{f_1} + 
		 b_i (x_i^{r_i} - x_j^{r_j}) \widetilde{f_2} = (x_i^{r_i} - x_j^{r_j})(a_i \widetilde{f_1} + b_i \widetilde{f_2}).$
		 Similarly, by (\ref{derivfa2}), for (\ref{derivationsB2}) derivations,
		 $$\eta_t = \frac{1}{r_{m+1}} x_{m+1}^{r_{m+1} + 1} \partial_{x_{m+1}} - \frac{1}{r_t} x_t^{r_t + 1} \partial_{x_t} 
		 \in \rom{Der}(a_i \widetilde{f_1} + b_i \widetilde{f_2}) \ \text{for} \ m+1 \le s< t \le n.$$
		 It follows that $$\delta_j, \eta_t \in \rom{Der}(\widetilde{f_1}) \cap \rom{Der}(\widetilde{f_2}) 
		 \bigcap_{i = 3}^k \rom{Der}(a_i \widetilde{f_1} + b_i \widetilde{f_2}) =\rom{Der}((\widetilde{f_1} \widetilde{f_2}\prod_{i = 3}^k (a_i \widetilde{f_1} + b_i \widetilde{f_2}))^{\rom{red}})$$ 
		 by Remark \ref{rmkDerIntersection} 
		 $\text{for} \ 1 \le  j \le m$, $\ m+1 \le s< t \le n$ 
		   and $k \ge 3$.
		   
		 Note that by
		 construction of $\widetilde{f_1}$ and $\widetilde{f_2}$, and since for a generic member $a_i \widetilde{S_1} + b_i \widetilde{S_2}$ of the pencil, $a_i$ and 
		 $b_i$ are non-zero, $\left(\widetilde{f_1}\widetilde{f_2}\prod_{i = 3}^k (a_i \widetilde{f_1} + b_i \widetilde{f_2})\right)^{\rom{red}} = 
		 (\widetilde{f_1} \widetilde{f_2})^{\rom{red}}\prod_{i = 3}^k (a_i \widetilde{f_1} + b_i \widetilde{f_2})$ and has weighted degree 
		 $knr_i\omega_i - (n+1)(r_i-1) + \sum_{i=0}^n\omega_i$.
		 
		 By the proof of (i), 
		 $\rom{Der}((\widetilde{f_1} \widetilde{f_2})^{\rom{red}})$ is free with exponents $(1, r_i\omega_i + 1, \ldots, r_i\omega_i+ 1)$ and the ME-scheme
		 $E(R_1, \ldots, S_n )$ has codimension $2$. This implies by Theorem \ref{propMEscheme} that 
		 $(\widetilde{f_1} \widetilde{f_2})^{\rom{red}} \in I_{E(R_1, \ldots, S_n )}$, which implies that 
		 $(\widetilde{f_1} \widetilde{f_2})^{\rom{red}}\prod_{i = 3}^k (a_i \widetilde{f_1} + b_i \widetilde{f_2}) \in I_{E(R_1, \ldots, S_n )}$ by definition of an ideal. By applying 
		 Theorem \ref{propMEscheme} again, we have that $V((\widetilde{f_1} \widetilde{f_2})^{\rom{red}}\prod_{i = 3}^k (a_i \widetilde{f_1} + b_i \widetilde{f_2}))$
		 is free with weighted exponents 
		 $(r_i\omega_i +1,\cdots, r_i\omega_i +1, d_n  )$, where
		 \begin{align*}
		 	d_n 
		 	&= \rom{deg}_\omega\left(\left(\widetilde{f_1}\widetilde{f_2}\right)^{\rom{red}}
		 	\prod_{i = 3}^k (a_i \widetilde{f_1} + b_i \widetilde{f_2})\right) + \sum_{i=0}^n (1-\omega_i) - (n-1)(r_i\omega_i +1) -1\\
		 	&=(k-2)nr_i\omega_i + 1.
		 \end{align*}
	 \end{proof}
	
	\begin{remark}
		Corollary \ref{corf1f2powerN} is a special case of Theorem \ref{thmf1f2weighted} when $r_i = r$ for all $i \in \{0, \ldots, n\}$. Since 
		$\widetilde{f_1}$ and $\widetilde{f_2}$ are already homogeneous of the same degree in this case, $\mathbb{P}^n_\omega = \mathbb{P}^n$ with
		$\omega = (1, \ldots, 1)$, and hence making the substitution $r_i = r$ and $\omega_i = 1$, we recover the exponents of the free divisors in Corollary
		 \ref{corf1f2powerN}.
	\end{remark}

It follows from Lemma \ref{lemFreeWeightedEqualFreeAffine} that the modules of logarithmic derivations of the divisors in affine $(n+1)$-space with polynomial ring 
$\kk[x_0, \ldots, x_n]$ are also free,
hence the corollary follows. 
Let $V$ be the $\kk$-module generated by monomials of degree at most $\rom{max}(\rom{deg}\, \widetilde{f_1}, \rom{deg}\, \widetilde{f_2})$. The 
$\kk$-submodule generated by $\{ \widetilde{f_1}, \widetilde{f_2}\}$ corresponds to a $\mathbb{P}^1 \subset \rom{Proj}(V)$, and we denote this line by 
$L_{ \widetilde{f_1}, \widetilde{f_2}}$.

\begin{corollary} \label{corAff_f1f2}
	Choose $m \in \{0,1,\ldots, n-1\}$ and let $r_0, \ldots, r_n$ be any positive integers; this defines
	$\widetilde{f_1}$ and $ \widetilde{f_2}$ as constructed in (\ref{eqf1f2tilde}).
	\begin{enumerate}
		\item The divisor $(\widetilde{S_1}\widetilde{S_2})^{\rom{red}}$ in affine space $\mathbb{A}^{n+1}$ defined by the polynomial
		$$(\widetilde{f_1}\widetilde{f_2})^{\rom{red}} = x_0 \cdots x_n 
		\left( \sum_{i = 0}^m x_{0}^{r_{0}} \cdots \widehat{x_i^{r_i}}\cdots x_m^{r_m} \right)
		\left( \sum_{i = m+1}^n x_{m+1}^{r_{m+1}} \cdots \widehat{x_i^{r_i}}\cdots x_n^{r_n} \right)$$
		is free.
		\item The divisor $(\widetilde{S_1}\widetilde{S_2})^{\rom{red}} \prod_{i = 3}^k (a_i \widetilde{S_1} + b_i \widetilde{S_2})$ in affine space $\mathbb{A}^{n+1}$ 
		defined by the polynomial
		$$\left(\widetilde{f_1}\widetilde{f_2}\right)^{\rom{red}}\prod_{i = 3}^k (a_i \widetilde{f_1} + b_i \widetilde{f_2}),$$
		for $k \ge 3$, where $a_i \widetilde{S_1} + b_i \widetilde{S_2} = V(a_i \widetilde{f_1} + b_i \widetilde{f_2}),$ with $[a_i;b_i] \in  \mathbb{P}^1$ for 
		$\ 3 \le i \le k$ are $k-2$ hypersurfaces in $\mathbb{A}^{n+1}$ corresponding to general distinct points of $L_{ \widetilde{f_1}, \widetilde{f_2}}$, is
		 free.\\		
	\end{enumerate}
\end{corollary}

By the \textit{cone construction}, one can projectivize and add the hyperplane at infinity $\{ x_{n+1} = 0\}$. Let $\_^h$  is homogenization in the variable $x_{n+1}$
 in $\kk[x_0, ... , x_{n+1}]$. 
The corresponding divisors are free in projective $n$-space as shown by the following theorem.

\begin{theorem} \label{thmDiGProj} Choose $m \in \{0,1,\ldots, n-1\}$ and let $r_0, \ldots, r_n$ be positive integers such that 
	$r_0 \le r_1 \le \ldots \le r_n$; this defines
	$\widetilde{f_1}$ and $ \widetilde{f_2}$ as constructed in (\ref{eqf1f2tilde}) with $\rom{deg}\, \widetilde{f_2}^h \ge \rom{deg}\, \widetilde{f_1}^h$.  Let the 
	hyperplane $H = V(x_{n+1}) \subset \mathbb{P}^{n+1}.$
	
	\begin{enumerate}
		\item The divisor $H((\widetilde{S_1}\widetilde{S_2})^{\rom{red}})^h$ in projective space $\mathbb{P}^{n+1}$ defined by the
		homogeneous polynomial $x_{n+1} ((\widetilde{f_1}\widetilde{f_2})^{\rom{red}})^h =$
			$$x_0 \cdots x_{n+1}
			\left( \sum_{i = 0}^m x_{n+1}^{r_i - r_0} x_{0}^{r_{0}} \cdots \widehat{x_i^{r_i}}\cdots x_m^{r_m} \right)
			\left( \sum_{i = m+1}^n x_{n+1}^{r_i - r_{m+1}} x_{m+1}^{r_{m+1}} \cdots \widehat{x_i^{r_i}}\cdots x_n^{r_n} \right)$$
		is free with exponents $(1,1,r_1 +1, r_2+1, \ldots,  \widehat{r_{m+1} +1}, \ldots, r_n +1)$.
		
		\item The divisor $H((\widetilde{S_1}\widetilde{S_2})^{\rom{red}})^h \prod_{i = 3}^k \left( a_i \widetilde{S_1} + b_i \widetilde{S_2} \right)^h$ in projective 
		space $\mathbb{P}^{n+1}$ defined by the homogeneous polynomial
		$$x_{n+1} ((\widetilde{f_1}\widetilde{f_2})^{\rom{red}})^h\prod_{i = 3}^k \left( a_i \widetilde{f_1} + b_i \widetilde{f_2} \right)^h,$$
		for $k \ge 3$, $[a_i;b_i] \in  \mathbb{P}^1$ and 
		$\left( a_i \widetilde{S_1} + b_i \widetilde{S_2} \right)^h= V \left(\left( a_i \widetilde{f_1} + b_i \widetilde{f_2} \right)^h \right)$ are generic members
		of the pencil $C(x_{n+1}^{\rom{deg} \widetilde{f_2} - \rom{deg} \widetilde{f_1}} \cdot \widetilde{f_1}^h, \widetilde{f_2}^h )$, is free with
		exponents $$(1,r_1 +1, r_2+1, \ldots,  \widehat{r_{m+1} +1}, \ldots, r_n +1, (k-2)\sum_{i = 1}^n r_i + 1 ).$$\\
		
	\end{enumerate}
\end{theorem}

\begin{proof}Denote by $R:= \kk[x_0, \ldots, x_n]$ and $S := \kk[x_0, \ldots, x_n, x_{n+1}].$ Let $\delta_{\mathcal{E}} = \sum_{i = 0}^{n+1} x_i\partial_{x_i}$ be the Euler derivation.
	Consider the $n$-derivations in the proof of Theorem \ref{thmf1f2weighted}: in (\ref{derivationsB1}), (\ref{derivationsB2}),
	\begin{align*}
		\delta_{\mathcal{E}_\omega} &= \sum^n_{i = 0} \omega_i x_i \partial_{x_i}\ \text{with} \ \omega_i = \prod_{j \ne i} r_j, \ \text{for} \  i = 0, \ldots, n,\\
		\delta_j &= \frac{1}{r_0} x_0^{r_0 + 1} \partial_{x_0} - \frac{1}{r_j} x_j^{r_j + 1} \partial_{x_j} \ \text{for} \ 1 \le  j \le m,\\
		\eta_t &= \frac{1}{r_{m+1}} x_{m+1}^{r_{m+1} + 1} \partial_{x_{m+1}} - \frac{1}{r_t} x_t^{r_t + 1} \partial_{x_t} \ \text{for} \ m+2 \le t \le n.
	\end{align*}
	Since $(\widetilde{f_1}\widetilde{f_2})^{\rom{red}} \in R$ and because the $n$-derivations, which we will denote by $\zeta_i$ for $i = 1, \ldots, n$, are in
	 $\rom{Der}((\widetilde{f_1}\widetilde{f_2})^{\rom{red}})$, by Corollary \ref{corHomogenizingDerv} this implies that for each homogenized derivation, 
	 $\zeta_i^h \in \rom{Der}(((\widetilde{f_1}\widetilde{f_2})^{\rom{red}})^h)$. It follows that $\zeta_i^h \in \rom{Der}(x_{n+1} ((\widetilde{f_1}\widetilde{f_2})^{\rom{red}})^h)$ for all 
	  $i \in \{ 1, \ldots, n\}$ since they all have zero coefficients for $\partial_{x_{n+1}}$.\\
	  
	  (i) From (\ref{derivationLastmu}), $\mu = \sum_{i = m+1}^n \frac{1}{r_i} x_i \partial_{x_i}$ is in $\rom{Der}((\widetilde{f_1}\widetilde{f_2})^{\rom{red}}) \subset \rom{Der}_{\kk}(R)$ and
	  from the proof of Theorem \ref{thmf1f2weighted}(i), we have seen that $\{ \mu,\zeta_1, \ldots, \zeta_n\}$ is a basis for
	   $\rom{Der}((\widetilde{f_1} \widetilde{f_2})^{\rom{red}}).$ By Corollary \ref{corAff_f1f2}(i),
	   $\rom{Der}((\widetilde{f_1} \widetilde{f_2})^{\rom{red}})$ is a free $R$-module, and a direct computation shows that\\ 
	   $\rom{deg} \,\mu + \sum^n_{i= 1} \rom{deg} \,\zeta_i = \rom{deg}((\widetilde{f_1} \widetilde{f_2})^{\rom{red}}).$ So by Proposition \ref{propConeConstructionFree},  
	   	$\rom{Der}(x_{n+1} \cdot( (\widetilde{f_1} \widetilde{f_2})^{\rom{red}})^h)$
	   is free with exponents $$(1,1,r_1 +1, r_2+1, \ldots,  \widehat{r_{m+1} +1}, \ldots, r_n +1),$$
	   which completes the proof of (i).\\
	  
	   (ii) Consider the $n$ derivations $\zeta_1, \ldots, \zeta_n$ defined at the beginning of the proof. In Theorem \ref{thmf1f2weighted}(ii)'s proof, we have seen that $\zeta_i \in
	   \rom{Der}((\widetilde{f_1} \widetilde{f_2})^{\rom{red}}\prod_{i = 3}^k (a_i \widetilde{f_1} + b_i \widetilde{f_2}))$ $\text{for}\  i = 1, \ldots, n$ and $k \ge 3$.
	   Hence by Corollary \ref{corHomogenizingDerv} and the arguments at the beginning, we conclude that $$\zeta_i^h \in
	   \rom{Der}\left(x_{n+1} \left(\left(\widetilde{f_1} \widetilde{f_2} \right)^{\rom{red}}\prod_{i = 3}^k (a_i \widetilde{f_1} + b_i \widetilde{f_2} )\right)^h \right)$$ 
	   $\text{for}\  i = 1, \ldots, n$ and $k \ge 3$. From Theorem \ref{thmf1f2weighted}(ii)'s proof that for a generic member of the pencil, $(a_i \widetilde{f_1} + b_i \widetilde{f_2} )$ 
	   is a reduced polynomial,  
	   and hence observe that 
	   	$$x_{n+1} \left(\left(\widetilde{f_1} \widetilde{f_2} \right)^{\rom{red}}\prod_{i = 3}^k (a_i \widetilde{f_1} + b_i \widetilde{f_2} )\right)^h \\
	   	=x_{n+1} \left(\left(\widetilde{f_1} \widetilde{f_2} \right)^{\rom{red}} \right)^h\prod_{i = 3}^k \left(a_i \widetilde{f_1} + b_i \widetilde{f_2} \right)^h$$
	   	is a reduced polynomial.
	   Let $T_i = ( g_0^i, g_1^i,\ldots ,g_{n+1}^i) \in (Sym^{d_i} \, k^{n+2})^{\oplus {n+2}}$ be the partially symmetric tensor corresponding to the homogeneous derivation
	   $\zeta_i^h$ for $i = 1, \ldots, n$ ordered by smallest degree, where $(d_1, d_2, \ldots, d_n) = (1,r_1 +1, r_2+1, \ldots,  \widehat{r_{m+1} +1}, \ldots, r_n +1)$ as
	    computed in (i), and consider the ME-scheme $E(T_1, \ldots, T_{n})$. 
	   Moreover, by the proof in part (i), 
	   $\rom{Der}(x_{n+1} ( (\widetilde{f_1} \widetilde{f_2})^{\rom{red}})^h)$ is free and $\zeta_i^h$
	    for $i = 1, \ldots, n$ are also part of a basis for
	   $\rom{Der}(x_{n+1} ( (\widetilde{f_1} \widetilde{f_2})^{\rom{red}})^h)$, hence $E(T_1, \ldots, T_{n})$ has codimension $2$. By
	    \cite[Proposition $3.3$]{digennaro2025saitostheoremrevisitedapplication}, this implies that
	    $x_{n+1} ( (\widetilde{f_1} \widetilde{f_2})^{\rom{red}})^h \in I_{E(T_1, \ldots, T_{n})}$, and this further implies that 
	    $$x_{n+1} ( (\widetilde{f_1} \widetilde{f_2})^{\rom{red}})^h \prod_{j = 3}^k (a_j \widetilde{f_1} + b_j \widetilde{f_2})^h 
	    \in I_{E(T_1, \ldots, T_{n})}$$ by definition of an ideal. An application of \cite[Proposition $3.3$]{digennaro2025saitostheoremrevisitedapplication} again now yields that 
	    $$\rom{Der}(x_{n+1} ( (\widetilde{f_1} \widetilde{f_2})^{\rom{red}})^h \prod_{j = 3}^k (a_j \widetilde{f_1} + b_j \widetilde{f_2})^h )$$ 
	    is free with exponents
	    $$(1,r_1 +1, r_2+1, \ldots,  \widehat{r_{m+1} +1}, \ldots, r_n +1, d_{n+1}).$$ Finally, we compute the degree $d_{n+1}$ of the last remaining derivation.
		Let $$F' = x_{n+1} ( (\widetilde{f_1} \widetilde{f_2})^{\rom{red}})^h \prod_{j = 3}^k (a_j \widetilde{f_1} + b_j \widetilde{f_2})^h ,$$ 
		so $\rom{deg}\, F' = (k-2)\sum_{i = 1}^n r_i + (n+2)+\sum_{i = 1}^m r_i + \sum_{j = m+2}^n r_j .$
		\begin{align*}
			\text{Therefore} \ d_{n+1} &= \rom{deg}\, F' - (\sum_{i = 1}^m (r_i + 1) + \sum_{j = m+2}^n (r_j + 1) + 2)\\  
			&= (k-2)\sum_{i = 1}^n r_i + 1.
		\end{align*}
\end{proof}

\begin{remark}
	Similar to Remark \ref{rmkBasesAffLiftRefArr}, while in the affine case $\mathbb{A}^{n+1}$ (and weighted projective space $\mathbb{P}^{n}_\omega$) one could fix any 
	$i \in \{ 0, \ldots, m\}$ and any $s \in \{ m+1, \ldots, n \}$ for derivations in (\ref{derivationsB1}) and (\ref{derivationsB2}) respectively to form part of a basis of 
	$\rom{Der}((\widetilde{f_1} \widetilde{f_2})^{\rom{red}})$ or $\rom{Der}(( \widetilde{f_1}\widetilde{f_2})^{\rom{red}}\prod_{i = 3}^k (a_i \widetilde{f_1} + b_i \widetilde{f_2}))$, 
	this is no longer true in the projective case $\mathbb{P}^{n+1}$ via the cone construction as shown in the proof of Theorem \ref{thmDiGProj} (i): choosing to 
	fix $i$ and $s$ different from $i=0$ and $s= m+1$ will
	in general not give derivations that are part of a basis that can be lifted (see Proposition \ref{propConeConstructionFree}) since the sum of degrees of the derivations of 
	the lifted bases will not be equal to the degrees of the reduced polynomials of Theorem \ref{thmDiGProj} (i) and (ii).
\end{remark}

\nocite{*}
\bibliographystyle{rendiconti}
\bibliography{bibliography}
\end{document}